\documentclass[11pt, a4paper]{article}
\usepackage[utf8]{inputenc}
\usepackage{mathtools}
\usepackage{amsmath}
\usepackage{amsthm}
\usepackage{algorithm}
\usepackage{algorithmic}
\usepackage{amssymb}
\usepackage{parskip}
\usepackage{epsfig}
\usepackage[colorlinks, citecolor=hanblue,linkcolor=mordantred19]{hyperref}

\usepackage{stackengine}
\mathtoolsset{showonlyrefs}
\usepackage{xcolor}
\usepackage[a4paper, margin=2.5cm]{geometry}
\usepackage{comment}
\usepackage{mathrsfs} 
\usepackage[export]{adjustbox}
\DeclareMathOperator*{\argmax}{arg\,max}
\DeclareMathOperator*{\argmin}{arg\,min}
\usepackage{multirow}

\newcommand{\mres}{\mathbin{\vrule height 1.4ex depth 0pt width
0.13ex\vrule height 0.13ex depth 0pt width 1.0ex}}

\usepackage{color}
\definecolor{hanblue}{rgb}{0.27, 0.42, 0.81}
\definecolor{mordantred19}{rgb}{0.68, 0.05, 0.0}
\definecolor{red}{rgb}{0.8, 0.0, 0.0}
\definecolor{green}{rgb}{0.0, 0.5, 0.0}

\DeclareMathOperator{\sign}{sign}
\renewcommand{\phi}{\varphi}

\DeclareMathOperator{\supp}{supp}
\DeclareMathOperator{\id}{id}

\usepackage{mathtools}
\usepackage{hyperref}

\newcommand{\R}{\mathbb{R}}

\newcommand{\N}{\mathbb{N}}
\newcommand{\M}{\mathcal{M}}

\newcommand{\KR}{\operatorname{KR}}
\newcommand{\TV}{\operatorname{TV}}
\newcommand{\opt}{\operatorname{opt}}
\newcommand{\Ext}{\operatorname{Ext}}
\newcommand{\Lip}{\operatorname{Lip}}
\newcommand{\dd}{\, \mathrm{d}}
\newcommand{\W}{\mathcal{W}}
\newcommand{\ab}{{\alpha, \beta}}

\newcommand\restr[2]{{
  \left.\kern-\nulldelimiterspace 
  #1 
  \vphantom{\big|} 
  \right|_{#2} 
  }}

\usepackage{enumitem}

\theoremstyle{plain}

\newtheorem{thm}{Theorem}[section]

\newtheorem{lemma}[thm]{Lemma}
\newtheorem{prop}[thm]{Proposition}

\newtheorem{ass}{Assumption}
\newtheorem{rem}[thm]{Remark}

\theoremstyle{definition}

\theoremstyle{remark}

\newtheorem{ex}[thm]{Example}

\definecolor{hanblue}{rgb}{0.27, 0.42, 0.81}

\usepackage{scalerel}

\title{Extremal points and sparse optimization for generalized Kantorovich-Rubinstein norms}
\author{Marcello Carioni$^\ast$, Jos\'e A. Iglesias\thanks{Department of Applied Mathematics, University of Twente, 7500AE Enschede, The Netherlands \newline (\texttt{m.c.carioni@utwente.nl}, \texttt{jose.iglesias{@}utwente.nl})} , Daniel Walter\thanks{Institut f\"ur Mathematik, Humboldt-Universit\"at zu Berlin, 10117 Berlin, Germany \newline (\texttt{daniel.walter@hu-berlin.de})}}
\date{}

\begin{document}

\maketitle

\begin{abstract}
A precise characterization of the extremal points of sublevel sets of nonsmooth penalties provides both detailed information about minimizers, and optimality conditions in general classes of minimization problems involving them. Moreover, it enables the application of accelerated generalized conditional gradient methods for their efficient solution. In this manuscript, this program is adapted to the minimization of a smooth convex fidelity term which is augmented with an unbalanced transport regularization term given in the form of a generalized Kantorovich-Rubinstein norm for Radon measures. More precisely, we show that the extremal points associated to the latter are given by all Dirac delta functionals supported in the spatial domain as well as certain dipoles, i.e., pairs of Diracs with the same mass but with different signs. Subsequently, this characterization is used to derive precise first-order optimality conditions as well as an efficient solution algorithm for which linear convergence is proved under natural assumptions. This behaviour is also reflected in numerical examples for a model problem.
\end{abstract}

\section{Introduction}
We consider minimization problems of the type
\begin{equation}\label{eq:introProb}\inf_{\mu \in \mathcal{M}(\Omega)} F(K\mu)+\|\mu\|_{\KR_p^\ab},\end{equation}
where $K : \mathcal{M}(\Omega)\rightarrow Y$ is a linear operator mapping the space $\M(\Omega)=(\mathcal{C}_0(\Omega))^\ast$ of signed finite Radon measures to a Hilbert space $Y$, $F : Y \rightarrow \R$ is a convex discrepancy and $\|\cdot\|_{\KR_p^\ab}$ is the generalized Kantorovich-Rubinstein norm defined as
\begin{equation}\label{eq:def-KRbp}
	    \|\mu\|_{\KR_p^\ab} := \inf_{\nu} W_p(\nu_+, \nu_-) + \frac{\beta}{2} |\nu|(\Omega) + \textbf{1}_{\nu(\Omega) = 0} + \alpha|\mu - \nu|(\Omega)
\end{equation}
for signed Radon measures $\mu,\nu \in \mathcal{M}(\Omega)$ on $\Omega$, a convex subset of $\R^n$. Here, $W_p$ is a Wasserstein distance (see \eqref{eq:defwass} below for the precise notation we use), so the $\KR_p^\ab$ norm reflects a form of unbalanced optimal transport with weights $\alpha >0$ for mass creation or destruction and $\beta \geq 0$ representing an additional penalization for the transported mass. Taking this into account, \eqref{eq:introProb} can be seen as a bilevel optimization problem with transport costs in the lower level.

It is immediate that definition \eqref{eq:def-KRbp} is only meaningful if $2\alpha - \beta > 0$, since otherwise the infimum would be attained at $\nu=0$ irrespective of $\mu$ and the norm would reduce to the total variation $\|\mu\|_{\TV} = |\mu|(\Omega)$. In this case 
\eqref{eq:introProb} becomes
\begin{equation}\label{eq:TVprob}\inf_{\mu \in \mathcal{M}(\Omega)} F(K\mu)+\alpha\|\mu\|_{\TV},\end{equation}
which has recently enjoyed great popularity as regularized inversion of the linear operator $K$, enforcing sparsity in an infinite dimensional setting \cite{candes2014towards, duval2015exact, bredies2013inverse}. Let us also point out that there have been previous works using the Kantorovich-Rubinstein norm $\|\cdot\|_{\KR_p^{\alpha, 0}}$ in inverse problems as a fidelity term, used alongside a higher order regularization term. In \cite{LelLorSchVal14} it was used along with total (gradient) variation regularization for image denoising and cartoon-texture decomposition while \cite{MetEtAl16} considers applications in geophysics, in both cases motivated by recovering oscillating signals to which the $\KR$ norms can assign low values through the transport term. 

We point out that the bilevel character of Problem~\eqref{eq:introProb} poses challenges, both, in its analysis as well as in its practical solution. However, motivated by recent results on nonsmooth, convex and one-homogeneous regularizers, these difficulties can be greatly alleviated once the extremal points of the sublevel sets of the Kantorovich-Rubinstein norm are available. It has been observed that an explicit description of such extremal points provides information on the structure of the sparse solutions of a regularized inverse problem when the observation is finite dimensional \cite{bredies2020sparsity, boyer2019representer, unser2017splines}. Moreover, it allows to devise accelerated generalized conditional gradient algorithms \cite{linconv}, i.e. infinite dimensional versions of the classical Frank-Wolfe algorithm \cite{dunn1978conditional, dunn1980convergence, frank1956algorithm, yu2017generalized} that are based on the iterative construction of linear combination of extremal points, converging to a solution of the minimization problem \cite{IglWal22, bredies2022generalized, duval2021dynamical, bredies2021extremal}.
These methods and algorithms are applicable to Problem~\eqref{eq:introProb} and they allow to formulate an optimization procedure that does not entail an inner minimization anymore. The main result of this paper, Theorem~\ref{thm:KRextremals}, gives a precise characterization of this set. Subsequently, we use these new-found extremal points to formulate simple first-order necessary and sufficient optimality conditions for~\eqref{eq:introProb} as well as to derive an efficient solution algorithm based on the accelerated generalized conditional gradient method presented in \cite{linconv}, see Algorithm~\ref{alg:accgcg} as well as Theorem~\ref{thm:sublinear} and~\ref{thm:linear}, respectively.

The minimization problem \eqref{eq:introProb} has the formal structure of regularized inversion of $K$. However, since the Kantorovich-Rubinstein norms can assign low values to self-similar oscillating signals, they are likely not advantageous by themselves as general-purpose regularization. Therefore, the potential applications we consider here are instead inspired by sparse optimal design incorporating awareness of a priori location information through a reference measure $\mu_r$, in which optimal transport can have a modelling motivation by itself. Moreover, since problems using the $\KR$ norm are computationally quite involved, the sparse optimization approach based on extremal points that we present here is likely to have applications beyond the particular minimization problems we use as examples.

\subsection{Preliminaries and notation}
Throughout, for nonnegative measures $\mu,\nu \in \M_+(\Omega) \subset \M(\Omega)$ we use the convention
\begin{equation}\label{eq:defwass}
	W_p(\mu, \nu) := \inf \left\{ \int_{\Omega \times \Omega} |x-y|^p \dd \gamma(x, y) \, \middle \vert\,  \gamma \in \Pi(\mu, \nu) \right\},
\end{equation}
where in particular we do not take the $1/p$-th power of the integrals. We focus on $p \in (0,1]$ and in this case we have that $(x,y)\mapsto |x-y|^p$ is a metric, by the subadditivity of increasing concave functions (see \cite[Lem.~2.1]{PegSanPia15}, for example). This also implies the ``metric'' triangle inequality $W_p(\mu, \nu) \leq W_p(\mu,\rho) + W_p(\rho,\nu)$, see \cite[Cor.~3.2]{San15}. In \eqref{eq:defwass}, $\Pi(\mu, \nu)$ denotes the usual set of couplings or transportation plans, that is, nonnegative measures $\gamma$ on $\Omega \times \Omega$ such that $(\pi_1)_\# \gamma = \mu$ and $(\pi_2)_\# \gamma = \nu$ for the pushforwards through the projections onto the first and second components, that is $\pi_1(x,y)=x$ for $x,y\in\Omega$, $(\pi_1)_\# \gamma(A) = \gamma(\pi_1^{-1}(A))$ for $A \subseteq \Omega$, and similarly for $\pi_2$.

The notation $\KR$ in \eqref{eq:def-KRbp} stands for Kantorovich and Rubinstein who introduced it for balanced measures, that is, those with $\mu(\Omega)=0$. A thorough treatment for that case can be found in \cite[Sec.~VIII.4]{KanAki82}. The case for unbalanced measures in the form \eqref{eq:def-KRbp} but with $\beta=0$ was introduced in \cite{Han92}.

A point $u$ belonging to a convex set $C$ is called \emph{extremal} if it cannot be written as a nontrivial convex combination of points in $C$, that is
\[u=\lambda v_1 + (1-\lambda)v_2 \text{ with } \lambda \in (0,1) \text{ and } v_1, v_2 \in C \text{ implies }u=v_1=v_2.\]
We denote the set of such points as $\Ext(C)$.

\subsection{Outline of the article}
In Section \ref{sec:extremals} we characterize the extremal points of the unit ball of the $\KR_p^\ab$ norm in $\mathcal{M}(\Omega)$. Section \ref{sec:minimizationproblem} is dedicated to Problem \eqref{eq:introProb}, and in it we discuss its first-order optimality conditions, introduce an accelerated generalized conditional gradient (AGCG) approach for its solution, and state results on sublinear and linear convergence. In Section \ref{sec:examples} we present numerical computations on a proof-of-concept instance of \eqref{eq:introProb} which demonstrate both the effect of the generalized $\KR$ norm, and the adequacy of the AGCG approach for its minimization. Appendix \ref{sec:appendix} contains the arguments needed to link Problem \eqref{eq:introProb} and the assumptions made in Section \ref{sec:minimizationproblem} to recent works on convergence of this type of method, which proves the stated convergence results.

\section{The \texorpdfstring{$\KR_p^\ab$}{KR} norm and its associated extremal points}\label{sec:extremals}
We have referred to $\|\cdot\|_{\KR_p^\ab}$ as a norm, but from our definitions it is not immediate that it is one. To check the triangle inequality for $\|\cdot\|_{\KR_p^\ab}$ it is enough to verify that 
\begin{equation}\label{eq:KRtriangle}
W_p\big((\nu+\mu)^+, (\nu+\mu)^-\big) \leq W_p(\nu^+, \nu^-) + W_p(\mu^+, \mu^-)
\end{equation}
for all balanced $\mu,\nu$. For the sake of completeness, we briefly prove it in the next lemma.
\begin{prop}\label{prop:triangle}
For all balanced measures $\mu,\nu$ the bound \eqref{eq:KRtriangle} holds true.
\end{prop}
\begin{proof}
Consider the duality formula for distance costs \cite[Sec.~3.1.1]{San15}, i,.e., 
\begin{equation}\label{eq:Wpduality}
	W_p(\rho_1, \rho_2) = \sup \left\{ \int_\Omega \psi(x) \dd(\rho_1-\rho_2)(x) \, 
 \middle\vert\, \psi \in \Lip_1(\Omega, p)\right\},
\end{equation}
where $\Lip_1(\Omega, p)$ stands for $1$-Lipschitz functions with respect to the metric $(x,y)\mapsto |x-y|^p$. This tells us, in particular, that $W_p(\rho_1, \rho_2)$ depends only on the difference $\rho_1-\rho_2$. Therefore, if for balanced measures $\mu, \nu$ we have 
\begin{equation}
	\nu^+ + \mu^+ = (\nu+\mu)^+ + \eta \ \text{ and } \nu^- + \mu^- = (\nu+\mu)^- + \eta,
\end{equation}
denoting by $\gamma_\mu$ and $\gamma_\nu$ optimal plans for $W_p(\nu^+, \nu^-)$ and $W_p(\mu^+, \mu^-)$, we obtain
\begin{align*}
	W_p\big((\nu+\mu)^+, (\nu+\mu)^-\big) &= W_p\big((\nu+\mu)^+ \!+\eta, \,(\nu+\mu)^- \!+\eta\big) \\
	&\leq \int_{\Omega \times \Omega} |x-y|^p \dd(\gamma_\nu+\gamma_\mu)(x,y) \\
	&= W_p(\nu^+, \nu^-) + W_p(\mu^+, \mu^-).\qedhere
\end{align*}
\end{proof}


We note that these considerations are implicit in \cite{KanAki82}, \cite{Han92} and some other works treating $\KR$ norms, since their definition involves the set of measures $\gamma$ on $\Omega \times \Omega$ from which $\nu$ can be recovered as $\nu(A) = (\pi_1)_\# \gamma(A) - (\pi_2)_\# \gamma(A)$ for $A \subseteq \Omega$, which is larger than the set of couplings $\Pi(\nu^+, \nu^-)$. Starting with that definition the triangle inequality follows directly, but then the duality formula \eqref{eq:Wpduality} is needed to see that the corresponding infimum is in fact attained on $\Pi(\nu^+, \nu^-)$. We conclude this section showing that the infimal convolution in \eqref{eq:def-KRbp} is exact. For the sake of generality we decide to prove this statement for $\Omega \subset \R^n$ convex but not necessarily compact, which requires a slightly more technical proof to construct tight minimizing sequences.

\begin{lemma}\label{lem:krexact}
If $\beta>0$, the infimal convolution in \eqref{eq:def-KRbp} is always exact, that is, the infimum is attained for any $\mu$.
\end{lemma}
\begin{proof}
We can use the direct method of the calculus of variations, using the weak* topology of $\M(\Omega)$. Consider a minimizing sequence $\nu_n$ for \eqref{eq:def-KRbp}. Note that $\nu_n$ is uniformly bounded in total variation since $\beta>0$, so up to a subsequence we can assume it weak* converges to some limit measure. However, a priori this limit is not necessarily balanced. For this reason, we construct a modified sequence $\tilde \nu_n$ that is uniformly tight (which enforces the limit to be balanced, as can be seen testing the convergence with $\mathcal{C}_0$ extensions of constant functions), has at most the same total variation as $\nu_n$, and is still minimizing. 

First note that, by inner regularity of $\mu$, for every $n >0$ there exists a compact set $A_n \subset \Omega$ such that $|\mu|(\Omega \setminus A_n) \leq 1/n$. 
Note that since $\Omega$ is convex we can assume that also $A_n$ is convex for every $n$ by taking its convex envelope. Moreover, we can also assume that the sets $A_n$ are ordered, and
$A_m  \subset \operatorname{int} A_n$ whenever $m < n$.
Denote by $\pi^n : \Omega \rightarrow A_n$ the projection onto $A_n$ and note that $\pi^n(\Omega \setminus A_n) \cap A_m = \emptyset$ for every $n > m$ since $\pi^n(\Omega \setminus A_n) \subset \partial A_n$. Define then the sequence of measures
\begin{equation}
    \tilde \nu_n := \pi^{n}_\#\nu_n.
\end{equation}
Note that $\tilde \nu_n$ is compactly supported in $A_{n}$ and it is thus tight.
Moreover $\tilde \nu_n$ is still balanced, its total variation is bounded above by that of $\nu_n$ (the projections may induce cancellations, since $(\pi^{n})^{-1}(x)$ is a half-line when $x\in \partial A_n$) and $W_p(\tilde \nu^+_n , \tilde \nu^-_n) \leq W_p(\nu^+_n , \nu^-_n)$ since $\pi^{n}$ is $1$-Lipschitz. We now note that the fidelity term can be estimated as
\begin{equation}\begin{aligned}
    |\tilde \nu_n - \mu|(\Omega) & = |\tilde \nu_n - \mu|(A_{n-1}) + |\tilde \nu_n - \mu|(\Omega \setminus A_{n-1})  \\
    & \leq |\nu_n - \mu|(A_{n-1}) + |\tilde \nu_n|(\Omega \setminus A_{n-1}) + |\mu|(\Omega \setminus A_{n-1}) \\
    & \leq |\nu_n - \mu|(A_{n-1}) + |\nu_n|(\Omega \setminus A_{n-1}) + |\mu|(\Omega \setminus A_{n-1}) \\
    & \leq |\nu_n - \mu|(A_{n-1}) + |\nu_n - \mu|(\Omega \setminus A_{n-1}) + 2|\mu|(\Omega \setminus A_{n-1})\\
    & \leq |\nu_n - \mu|(\Omega) + \frac{2}{n-1}
\end{aligned}\end{equation}
since $|\mu|(\Omega \setminus A_{n-1}) \leq 1/(n-1)$ and $\pi^n(\Omega \setminus A_{n}) \cap A_{n-1} = \emptyset$.
This shows that $\tilde \nu_n$ is a minimizing sequence as well.
Since the total variation of a measure is weak* lower semicontinuous and the set of balanced measures is weak* closed it remains to show that $\nu \mapsto W_p(\nu^+,\nu^-)$ is weak* lower semicontinuous. This property follows again by the duality formula for distance costs \eqref{eq:Wpduality}, in which $W_p(\nu^+, \nu^-)$ is expressed as a supremum of weak* continuous functionals.
\end{proof}

\begin{lemma}\label{lem:lscKR}
The map $\mu \mapsto \|\mu\|_{\KR^\ab_p}$ is weak* lower semicontinuous.
\end{lemma}
\begin{proof}
We have defined $\|\cdot\|_{\KR^\ab_p}$ as the infimal convolution of $\W^\beta_p(\nu):=W_p(\nu^+, \nu^-) + \frac{\beta}{2}|\nu|(\Omega) + \textbf{1}_{\nu(\Omega) = 0}$ which is proper and bounded from below, and $|\cdot|(\Omega)$ which is proper, weak* lower semicontinuous and weak* coercive. In this situation it is enough (see \cite[Thm.~2.5(b)]{stromberg1996operation} or \cite[Prop.~6.5.5]{Lau72}) to notice that $\nu \mapsto \W^\beta_p(\nu)$ is weak* lower semicontinuous as well, which was proved in Lemma \ref{lem:krexact}. 
\end{proof}

\begin{lemma}\label{lem:compKR}
If $\alpha, \beta >0$, the set $\{ \mu \,|\, \|\mu\|_{\KR^\ab_p} \leq 1\}$ is the closed convex hull of the set of its extremal points, which is in particular nonempty.
\end{lemma}
\begin{proof}
By the straightforward estimate
\begin{equation}\label{eq:coercivity}\|\mu\|_{\KR^\ab_p} \geq \min \left(\alpha, \frac{\beta}{2}\right)\!|\mu|(\Omega) \quad\text{for all }\mu \in \M(\Omega)\end{equation}
and applying the Banach-Alaoglu theorem, the set in consideration is precompact with respect to the weak* topology of $\M(\Omega)$. Noting that by Lemma \ref{lem:lscKR} this set is also weak* closed and hence compact, we can apply the Krein-Milman theorem to obtain the claim.
\end{proof}

\subsection{Extremal points}
For the remainder of this section, our aim is to characterize the extremal points of the set $\{\mu \,\vert\, \|\mu\|_{\KR^{\alpha, \beta}_p} \leq 1\}$. In particular, we will see that they contain rescaled \emph{dipoles}, defined as
\begin{align}
 \mathcal{D}_\beta(x,y) := \frac{1}{\beta + |x-y|^p}(\delta_x - \delta_y)\ \text{ for }\ (x,y) \in \Omega \times \Omega\ \text{ with }\ x \neq y,
\end{align}
where we drop the dependence on $p$ for notational convenience.
These represent an elementary transport operation and arise from the $W_p$ term in the definition of $\|\cdot\|_{\KR_p^\ab}$. These same objects are known \cite[Ch.~3, Cor.~3.45]{Wea18} to be extremals for the unit ball of preduals of pointed Lipschitz spaces, corresponding to the simultaneous restriction to balanced measures for which $\mu(\Omega)=0$, and $\beta=0$. Moreover, in \cite[Thm.~6.1]{AngAscDonMan20} and \cite[Thm.~4.2]{AngAscMan21} atomic decompositions have been recently proved for Borel measures with respect to the $\|\cdot\|_{\KR_p^{1,0}}$ norm, obtaining that these can be expressed as a (countable) series of Dirac masses and dipoles. 

Here, we treat the unbalanced case with $\beta >0$, which makes the norm coercive with respect to the total variation. Interestingly, this total variation penalization does not alter the structure of the extremal points compared to the case $\beta=0$, but just their normalization. Our proofs are self-contained and use only well-known facts about the Kantorovich formulation of optimal transport, which makes the methods quite different from the works cited above.

The characterization we obtain turns out to be the following:
\begin{thm}\label{thm:KRextremals}
The set $\Ext\big\{ \mu\, \big \vert\, \|\mu\|_{\KR^\ab_p} \leq 1 \big\}$ of extremal points is precisely
\begin{equation}\label{eq:KRextremals}
	\left\{ \pm \frac{1}{\alpha}\delta_x \, \big \vert\, x \in \Omega \right\} \cup \left\{ \frac{1}{\beta +|x-y|^p}\big(\delta_x-\delta_y\big) \, \middle \vert\, x,y \in \Omega,\ 0 < |x-y|^p < 2\alpha - \beta \right\}.
\end{equation}
\end{thm}

\subsection{Balanced measures}
We start by first excluding the last term of \eqref{eq:def-KRbp}, which is equivalent to setting $\alpha = 0$ or restricting to the subspace of balanced measures.
\begin{prop}\label{prop:dipolesareext}
Assume $p \in (0,1)$. Then, rescaled dipoles of the form
\begin{equation}
	\frac{1}{|x-y|^p}(\delta_x - \delta_y)\ \text{ for }\ (x,y) \in \Omega \times \Omega\ \text{ with }\ x \neq y.
\end{equation}
are extremal points of the set
\begin{equation}\label{eq:theball}
	\mathcal{B}_p := \big\{ \mu \in \M(\Omega)\, \big \vert\, \mu(\Omega)=0, \  W_p(\mu^+, \mu^-) \leq 1 \big\}.
\end{equation}
\end{prop}
\begin{proof}
Assume we have for $\lambda \in (0,1)$ a convex combination 
\begin{equation}
	\frac{1}{|x-y|^p}(\delta_{x} - \delta_{y}) = \lambda \nu_1 + (1-\lambda) \nu_2,
\end{equation}
and let $\nu_1 = \nu_1^+ - \nu_1^-$ and $\nu_2 = \nu_2^+ - \nu_2^-$ be the Hahn decompositions of $\nu_i$. We have then
\begin{equation}\label{eq:dipoledecomp2}
	\frac{1}{|x-y|^p}(\delta_{x} - \delta_{y}) = \lambda \nu^+_1 - (1-\lambda) \nu^-_2 + (1-\lambda) \nu^+_2 - \lambda \nu^-_1.
\end{equation}
If we knew that $\lambda \nu^+_1 + (1-\lambda) \nu^+_2$ and $\lambda \nu^-_1 + (1-\lambda) \nu^-_2$ are the positive and negative parts $\delta_x / |x-y|^p$ and $\delta_y / |x-y|^p$ of the left hand side, we could conclude immediately. This is not true in general however, since there might be cancellations between the first two or last two terms of \eqref{eq:dipoledecomp2}. Let us denote the results of these partial sums as
\begin{equation}
	\xi_A := \lambda \nu^+_1 - (1-\lambda) \nu^-_2,\ \text{ and }\ \xi_B := (1-\lambda) \nu^+_2 - \lambda \nu^-_1.
\end{equation}
For these, we can also consider the Hahn decompositions $\xi_A = \xi^+_A - \xi^-_A$ and $\xi_B = \xi^+_B - \xi^-_B$, which give us that
\begin{equation}\begin{aligned}
	\lambda \nu^+_1 = \xi_A^+ + \eta_A \ \text{ and }\ (1-\lambda)\nu_2^- = \xi_A^- + \eta_A, \ \text{ for }\  \eta_A := \lambda \nu^+_1 - \xi_A^+ = (1-\lambda)\nu_2^- - \xi_A^-, \\
	\lambda \nu^-_1 = \xi_B^- + \eta_B \ \text{ and }\ (1-\lambda)\nu_2^+ = \xi_B^+ + \eta_B, \ \text{ for }\  \eta_B := \lambda \nu^-_1 - \xi_B^- = (1-\lambda)\nu_2^+ - \xi_B^+,
\end{aligned}\end{equation}
so that $\eta_A$ and $\eta_B$ are precisely the potential cancellations happening in \eqref{eq:dipoledecomp2}. 
In particular we have $\eta_A \in \mathcal{M}_+(\Omega)$, since given a measurable $C \subset \Omega$ such that $\eta_A(C) < 0$ we would have both
\begin{equation}
\xi_A^+(C) > \lambda \nu_1^+(C) \geq 0 \ \text{ and }\ \xi_A^-(C) > (1-\lambda) \nu_2^-(C)\geq 0
\end{equation}
contradicting the optimality of the Hahn decomposition $\xi_A = \xi^+_A - \xi^-_A$. A similar argument proves that  $\eta_B \in \mathcal{M}_+(\Omega)$. Moreover, since $\eta_A$ and $\eta_B$ appear with different signs in the decompositions $\nu_1^+, \nu_1^-$ and $\nu_2^+, \nu_2^-$ we necessarily have that they are singular to each other.

This implies as well that we can write
\begin{gather}
	\frac{1}{|x-y|^p}(\delta_{x} - \delta_{y}) = \lambda \tilde{\nu}_1 + (1-\lambda) \tilde{\nu}_2,\ \text{ where}\\
	\tilde{\nu}_1 := \nu_1 - \frac{1}{\lambda}\eta_A + \frac{1}{\lambda}\eta_B \ \text{ and }\ \tilde{\nu}_2 := \nu_2 + \frac{1}{1-\lambda}\eta_A - \frac{1}{1-\lambda}\eta_B.
\end{gather}
Now, in this decomposition there can be no cancellations between $\tilde{\nu}_1$ and $\tilde{\nu}_2$. Indeed, if we suppose that there exists a measurable set $C \subset \Omega$ such that $\tilde \nu_1(C) <0$ and $\tilde \nu_2(C)>0$ (or $\tilde \nu_1(C) <0$ and $\tilde \nu_2(C) >0$) the simple computation
\begin{equation}
\tilde \nu_1  = \frac{1}{\lambda} (\xi^+_A - \xi^-_B) \,, \ \text{ and }\  \tilde \nu_2  = \frac{1}{1-\lambda} (\xi_B^+ - \xi_A^-)\,
\end{equation}
contradicts the optimality of the Hahn decomposition $\xi_A = \xi^+_A - \xi^-_A$ (or the optimality of the Hahn decomposition $\xi_B = \xi^+_B - \xi^-_B$). Since no cancellations are happening between $\tilde \nu_1$ and $\tilde \nu_2$ there exist some $c_1^x, c_1^y, c_2^x, c_2^y >0$ with $\lambda c_1^x + (1-\lambda) c_2^x=1$ and $\lambda c_1^y + (1-\lambda)  c_2^y=1$ such that we have
\begin{equation}\begin{aligned}
	\nu_1^+ = \frac{c_1^x}{|x-y|^p}\delta_x + \frac{1}{\lambda}\eta_A, &\quad \nu_1^- = \frac{c_1^y}{|x-y|^p}\delta_y + \frac{1}{\lambda}\eta_B,\\
	\nu_2^+ = \frac{c_2^x}{|x-y|^p}\delta_x + \frac{1}{1-\lambda}\eta_B, &\quad \nu_2^- = \frac{c_2^y}{|x-y|^p}\delta_y + \frac{1}{1-\lambda}\eta_A.
\end{aligned}\end{equation}
Let us denote by $\gamma^{\opt}_i$ any optimal plans (see for example \cite[Thm.~1.4]{San15} for existence) for $W_p(\nu_i^+, \nu_i^-)$ with $i=1,2$. From the above expressions we observe that if $\eta_A \neq 0$ or $\eta_B \neq 0$, then necessarily
\begin{equation}\label{eq:sendoutsideeach}
	\gamma^{\opt}_i\big( \Omega \times \Omega \setminus \big ( \{(x,y)\} \cup \{(z,z) \,\big\vert\, z\in \Omega\} \big)\big) > 0.
\end{equation}
Indeed, since $\eta_A \perp \eta_B$ there holds that $\eta_A \neq \eta_B$.  Moreover, $x\notin \supp \eta_A$ and $y\notin \supp \eta_B$ due to the relations $\lambda c_1^x + (1-\lambda) c_2^x=1$ and $\lambda  c_1^y + (1-\lambda) c_2^y=1$ and \eqref{eq:dipoledecomp2}.
We then consider the convex combination transport plan
\begin{equation}
	\gamma_C := \lambda \gamma^{\opt}_1 + (1-\lambda) \gamma^{\opt}_2 \text{ with }\int_{\Omega \times \Omega} |z-w|^p \dd \gamma_C(z,w) = 1,
\end{equation}
for which
\begin{equation}
\gamma_C \in \Pi\big(\lambda \nu_1^+ + (1-\lambda) \nu_2^+,\ \lambda \nu_1^- + (1-\lambda) \nu_2^- \big) = \Pi\left(\frac{1}{|x-y|^p}\delta_x + \eta_A + \eta_B,\ \ \frac{1}{|x-y|^p}\delta_y + \eta_A + \eta_B \right),
\end{equation}
and because of \eqref{eq:sendoutsideeach} also
\begin{equation}\label{eq:sendoutside}
	\gamma_C\big( \Omega \times \Omega \setminus \big ( \{(x,y)\} \cup \{(z,z) \,\big\vert\, z\in \Omega\} \big)\big) > 0.
\end{equation}
Using the strict concavity of the cost (see \cite[Thm.~2.2]{PegSanPia15}) we obtain that any optimal plan $\gamma_0$ between $\delta_x/|x-y|^p + \eta_A + \eta_B$ and $\delta_y/|x-y|^p + \eta_A + \eta_B$ must leave $\eta_A$ and $\eta_B$ invariant. But this means that necessarily
\begin{equation}\label{eq:theuniqueplan}
	\gamma_0 = \delta_{(x,y)} + (\id, \id)_\#(\eta_A, \eta_B),
\end{equation}
which implies that 
\begin{equation}
	W_p\left(\frac{1}{|x-y|^p}\delta_x + \eta_A + \eta_B,\ \ \frac{1}{|x-y|^p}\delta_y + \eta_A + \eta_B \right)=1,
\end{equation}
but since the cost of $\gamma_C$ is also $1$ this means that the latter is also optimal and hence $\gamma_C = \gamma_0$, which leads to a contradiction with \eqref{eq:sendoutside} unless $\eta_A = \eta_B =0$.
\end{proof}

\begin{rem}\label{rem:unbounded}
Despite being described as a ball, the set $\mathcal{B}_p$ contains many directions which are unbounded in the natural total variation sense. One may take for example $\Omega = \overline{B(0,1)} \subset \R^n$, the rescaling by a factor $0<r<1$ defined by $S_r(x)=rx$ and the corresponding pushforwards $(S_r)_\# \nu^+$, $(S_r)_\# \nu^-$, and $(S_r, S_r)_\# \gamma$ for any $\gamma \in \Pi(\nu^+, \nu^-)$. In case $\nu$ is concentrated on finitely many points (so that the total mass of $(S_r)_\# \nu^+$ and $(S_r)_\# \nu^-$ is independent of $r$) we then have 
\begin{equation}
W_p(r^{-p}(S_r)_\# \nu^+, r^{-p}(S_r)_\# \nu^-) = W_p(\nu^+, \nu^-),
\end{equation}
which might lead to the intuition that $\mathcal{B}_p$ is more similar to a cone. However, this view is also not quite accurate, since the zero measure is not an extremal point for any $p \in (0,\infty)$. Indeed, for all $x\neq y$ the convex decomposition
\begin{equation}
    0 = \frac{1}{2}(\delta_x - \delta_y) + \frac{1}{2}(\delta_y - \delta_x)
\end{equation}
is always nontrivial.
\end{rem}

\begin{prop}\label{prop:onlydipolesext}
For all $p \in (0,1]$ and $\beta \in [0, \infty)$, any extremal point $\nu$ of
\begin{equation}
    \left\{ \mu \in \M(\Omega) \, \middle \vert\, \mu(\Omega)=0, \  W_p(\mu^+, \mu^-) + \frac{\beta}{2} |\mu|(\Omega) \leq 1 \right\}
\end{equation}
must be a rescaled dipole, that is, there exist points $x,y\in\Omega$ with $x \neq y$ for which $\nu = \mathcal{D}_\beta(x,y)$.
\end{prop}
\begin{proof}
Assume for the sake of contradiction that $\nu$ is not a dipole. Then, for the Hahn decomposition $\nu=\nu^ + - \nu^-$ either $\nu^+$ or $\nu^-$ has a support consisting of more than one point. Without loss of generality, we assume it is the former. Therefore, there exists a set $E \subset \Omega$ for which $\nu^+(E) >0$ and $\nu^+(F)>0$ for $F := \Omega \setminus E$ hold simultaneously, which induces a nontrivial decomposition
\begin{equation}
	\nu^+ = \nu^+ \mres E + \nu^+ \mres F.
\end{equation}
Now, let $\gamma_{\opt}$ be an optimal transportation plan for $W_p(\nu^+, \nu^-)$. We can use it to define  a ``pushforward measure'' of these sets by 
\begin{equation}\begin{aligned}
	\mu_E &:= (\pi_2)_\# \left[ \gamma_{\opt} \mres \big( (\pi_1)^{-1}(E) \big)\right],\ \text{ and}\\
	\mu_F &:= (\pi_2)_\# \left[ \gamma_{\opt} \mres \big( (\pi_1)^{-1}(F) \big) \right],
\end{aligned}\end{equation}
for which, using that $(\pi_1)_\# \gamma_{\opt} = \nu^+$, we have
\begin{equation}
	\mu_E(\Omega) = \left[ \gamma_{\opt} \mres \big( (\pi_1)^{-1}(E) \big)\right] \big(\Omega \times \Omega \big) = \gamma_{\opt}\big( (\pi_1)^{-1}(E) \big) = \nu^+(E) = \big( \nu^+ \mres E \big)(\Omega).
\end{equation}
Moreover, we notice that
\begin{equation}\begin{aligned}
	&\gamma_{\opt} \mres \big( (\pi_1)^{-1}(E) \big) \in \Pi \big( \nu^+ \mres E,\, \mu_E \big ), \text{ so that }\\
	&W_p(\nu^+ \mres E, \mu_E) \leq \int_{(\pi_1)^{-1}(E)} |z-w|^p \dd \gamma_{\opt}(z,w)
\end{aligned}\end{equation}
and similarly for $\mu_F$, and in fact
\begin{equation}\begin{aligned}
	&W_p(\nu^+ \mres E, \mu_E) + W_p(\nu^+ \mres F, \mu_F) +\beta |\nu^+ \mres E|(\Omega) + \beta |\nu^+ \mres F|(\Omega)\\
	&\qquad\leq\int_{(\pi_1)^{-1}(E)} |z-w|^p \dd \gamma_{\opt}(z,w) + \int_{(\pi_1)^{-1}(F)} |z-w|^p \dd \gamma_{\opt}(z,w) + \beta |\nu^+|(\Omega)\\
	&\qquad= W_p(\nu^+, \nu^-) +\frac{\beta}{2}|\nu|(\Omega)= 1.
\end{aligned}\end{equation}
With this in view let us define
\begin{equation}\begin{aligned}
    C_E := W_p(\nu^+ \mres E, \mu_E) &+\beta |\nu^+ \mres E|(\Omega) \in (0,1),\\
    C_F := W_p(\nu^+ \mres F, \mu_F) &+\beta |\nu^+ \mres F|(\Omega) \in (0,1),\\
    \nu_{1+} := \frac{C_E + C_F}{C_E}\,\nu^+ \mres E&,\quad \nu_{1-} := \frac{C_E + C_F}{C_E}\,\mu_E,\\
    \nu_{2+} := \frac{C_E + C_F}{C_F}\,\nu^+ \mres F&,\quad \nu_{2-} := \frac{C_E + C_F}{C_F}\,\mu_F.
\end{aligned}\end{equation}
For these, we have
\begin{equation}\begin{aligned}
	&\nu_{1+} \perp \nu_{1-},\ \ \nu_{1+}(\Omega) = \nu_{1-}(\Omega),\ \ W_p(\nu_{1+}, \nu_{1-}) + \beta |\nu_1|(\Omega) \leq 1,\text{ and}\\
	&\nu_{2+} \perp \nu_{2-},\ \ \nu_{2+}(\Omega) = \nu_{2-}(\Omega),\ \ W_p(\nu_{2+}, \nu_{2-}) + \beta |\nu_2|(\Omega) \leq 1,
\end{aligned}\end{equation}
but also 
\begin{equation}
	\nu = \lambda \big( \nu_{1+} - \nu_{1-} \big) + (1-\lambda) \big( \nu_{2+} - \nu_{2-} \big) \text{ for }\lambda := \frac{C_E}{C_E + C_F} \in (0,1),
\end{equation}
which is a contradiction with $\nu$ being extremal in \eqref{eq:theball}.
\end{proof}

\begin{ex}
Interestingly, dipoles are not extremal in case $p=1$. To see this, just consider $\Omega = [0,1]$, a number $a \in (0,2)$ and the measures
\begin{equation}\begin{aligned}
\nu_1 &:= \left(1-\frac{a}{2}\right)\delta_0 + a\, \delta_{1/2} - \left(1+\frac{a}{2}\right)\delta_1, \text{ and }\\
\nu_2 &:= \left(1+\frac{a}{2}\right)\delta_0 - a\, \delta_{1/2} - \left(1-\frac{a}{2}\right)\delta_1.
\end{aligned}\end{equation}
Then we have that
\begin{equation}
	W_1\left(\nu_1^+ ,\ \nu_1^- \right) = a \, \frac12 + \left(1 - \frac{a}{2} \right) = 1 = W_1\left(\nu_2^+ ,\ \nu_2^- \right),
\end{equation}
but also 
\begin{equation}
	\frac12 \nu_1 + \frac12 \nu_2 = \delta_0 - \delta_1.
\end{equation}
\end{ex}

In fact, this idea can be generalized to any measure and all $p \geq 1$:
\begin{prop}\label{prop:lackofextremals}
Assume that $\Omega$ is convex. Then if $p=1$ the set $\mathcal{B}_p$ of \eqref{eq:theball} has no nonzero extremal points, and if $p>1$ it is not convex.
\end{prop}
\begin{proof}
Let $\nu$ be any nonzero measure with $\nu(\Omega) = 0$. Then, if $\gamma_{opt}$ is optimal for $W_p(\nu^+, \nu^-)$ and since we have assumed $p \geq 1$, we can construct (see \cite[Thm.~5.27]{San15}) a constant-speed geodesic in $p$-Wasserstein space $[0,1] \ni t \mapsto \nu^t$ between $\nu^+$ and $\nu^-$ as $\nu^t := (\pi_t)_\# \gamma_{\opt}$, where $\pi_t(x,y)=(1-t)x+ty$. With it we then define for $a \in (0,2)$ the measures
\begin{equation}\begin{aligned}
\nu_1 &:= \left(1-\frac{a}{2}\right)\nu^+ + a\, \nu^{1/2} - \left(1+\frac{a}{2}\right)\nu^-, \text{ and }\\
\nu_2 &:= \left(1+\frac{a}{2}\right)\nu^- - a\, \nu^{1/2} - \left(1-\frac{a}{2}\right)\nu^-.
\end{aligned}\end{equation}
For the first of these, since $\nu^t$ is a constant-speed geodesic between $\nu^+$ and $\nu^-$, we have (see \cite[Box~5.2]{San15}, for example) that
\begin{equation}\big( W_p(\nu^{1/2}, \nu^-) \big)^{1/p} = \frac{1}{2},\end{equation} 
which for some $\gamma_{\opt}^{1/2}$ optimal for $W_p(\nu^{1/2}, \nu^-)$ allows us to define the transportation plan
\begin{equation}
    \gamma_a := \left(1-\frac{a}{2}\right)\gamma^{\opt} + a \gamma_{\opt}^{1/2} \in \Pi\left(\left(1-\frac{a}{2}\right)\nu^+ + a\, \nu^{1/2}, \left(1+\frac{a}{2}\right)\nu^-\right),
\end{equation}
so that
\begin{equation}\begin{aligned}\label{eq:slidingmasscost}
	W_p(\nu_1^+, \nu_1^-) &\leq \int_{\Omega \times \Omega} |x-y|^p \dd\gamma_a(x,y) \\&=  \left(1-\frac{a}{2}\right)W_p(\nu^+, \nu^-) + a W_p(\nu^{1/2}, \nu^-) \\ &= \left(1-\frac{a}{2}\right) + \frac{a}{2^p} \leq 1.
\end{aligned}\end{equation}
That $W_p(\nu_2^+, \nu_2^-) \leq 1$ is obtained entirely similarly, and we have expressed $\nu$ as the nontrivial convex combination $\nu = \nu_1/2 + \nu_2/2$. Notice that if $p>1$, because of the denominator in the second term of \eqref{eq:slidingmasscost} there is $r>1$ such that $r\eta_1,r\eta_2 \in \mathcal{B}_p$, so the decomposition tells us that this set is not convex.
\end{proof}

We have seen in Remark \ref{rem:unbounded} that with $\beta=0$ the norm \eqref{eq:def-KRbp} is not coercive in $\mathcal{M}(\Omega)$. For practical applications this would be quite unwieldy, hence we penalize the mass of the balanced part $\nu$. Interestingly, this does not alter the structure of the extremal points but just their normalization, which allows us to maintain a clean interpretation of them in terms of transport.

\begin{lemma}\label{lem:balancedtvext}
The extremal points of the set
\begin{equation}\label{eq:balancedtvball}
	\left\{ \mu \in \M(\Omega) \,\middle\vert\, \mu(\Omega)=0, \ \frac{\beta}{2}|\mu| \leq 1 \right\}
\end{equation}
are the dipoles
\begin{equation}
	\frac{1}{\beta} (\delta_x - \delta_y)\ \text{ for }\ (x,y) \in \Omega \times \Omega\ \text{ with }\ x \neq y.
\end{equation}
\end{lemma}
\begin{proof}
Assume without loss of generality that $\beta=1$. Arguing as in Proposition \ref{prop:dipolesareext}, assuming that $\delta_x - \delta_y = \lambda \nu_1 + (1-\lambda) \nu_2$ we can arrive to 
\begin{equation}\begin{aligned}
	\nu_1^+ = c_1^x\delta_x + \frac{1}{\lambda}\eta_A, &\quad \nu_1^- = c_1^y\delta_y + \frac{1}{\lambda}\eta_B,\\
	\nu_2^+ = c_2^x\delta_x + \frac{1}{1-\lambda}\eta_B, &\quad \nu_2^- = c_2^y\delta_y + \frac{1}{1-\lambda}\eta_A,
\end{aligned}\end{equation}
with $\lambda c_1^x + (1-\lambda) c_2^x=1$ and $\lambda  c_1^y + (1-\lambda) c_2^y=1$, $\eta_A \perp \eta_B$, $x\notin \supp \eta_A$ and $y\notin \supp \eta_B$. But if $|\eta_A|(\Omega)>0$ or $|\eta_B|(\Omega)>0$, then necessarily $|\nu_1|(\Omega)>2$ or $|\nu_2|(\Omega)>2$, so that $\delta_x -\delta_y$ is extremal in \eqref{eq:balancedtvball}.

For the converse, let us assume either $\nu^+$ or $\nu^-$ is supported in more than one point, say $\nu^+$ without loss of generality. Then as in Proposition \ref{prop:onlydipolesext} there would be sets $E,F$ for which
\begin{equation}
	\nu^+ = \nu^+ \mres E + \nu^+ \mres F \,\text{ and }\, \nu^+(E) + \nu^+(F) = 1,
\end{equation}
which can be rephrased as
\begin{equation}\label{eq:posconv}
	\nu^+ = \lambda \frac{\nu^+ \mres E}{\nu^+(E)} + (1-\lambda) \frac{\nu^+ \mres F}{\nu^+(F)} \,\text{ for }\, \lambda := \nu^+(E),
\end{equation}
which combined with $\nu^- = \lambda \nu^- + (1-\lambda) \nu^-$ gives us a nontrivial decomposition of $\nu$ within \eqref{eq:balancedtvball}.
\end{proof}

\begin{lemma}\label{lem:gaugeextremals}
Let $X$ be a Banach space, $f$ a convex positively one-homogeneous functional such that $f(u) = 0$ if and only if $u=0$, and $e \in X$. Then $e \in \Ext\{ w \mid f(w) \leq f(e) \}$ is equivalent to
\begin{equation}\label{eq:gaugestrictineq}\begin{gathered}
	f(e) < \lambda f(u) + (1-\lambda) f(v)\\ \text{whenever }\, e = \lambda u + (1-\lambda) v \, \text{ for }\, \lambda \in (0,1) \,\text{ and }\, u, v\notin \R^+ e \,\text{ with }\, u \neq v.
\end{gathered}\end{equation}
\end{lemma}
\begin{proof}
If condition \eqref{eq:gaugestrictineq} holds then it does so also in the particular case when $f(u)=f(v)$. Therefore, in that case it is not possible to have $u \neq v$ and $\lambda \in (0,1)$ with $f(u)=f(v)=f(e)$ and $e = \lambda u + (1-\lambda) v$, that is, $e$ must belong to $\Ext\{ w \mid f(w) \leq f(e) \}$. 

To prove the converse, let us assume that $e \in \Ext\{ w \mid f(w) \leq f(e) \}$ and $e = \lambda u + (1-\lambda) v$ for $\lambda \in (0,1)$, $u \neq v$ and $u,v\notin \R^+ e$. For this, we distinguish two cases.

The first case is when $f(u)=f(v)$. In this case, by convexity either $f(e)<f(u)=f(v)$ in which case the inequality of \eqref{eq:gaugestrictineq} follows immediately, or 
$f(e)=f(u)=f(v)$ which is not possible since it would lead to a contradiction with $e \in \Ext\{ w \mid f(w) \leq f(e) \}$.

The second case is when $f(u) \neq f(v)$, which we can try to reduce to the first case by rescaling $u$ and $v$ to $r_u u$ and $r_v v$ respectively, for some positive factors $r_u, r_v$. Note that $f(u), f(v) \neq 0$, since otherwise $u \in \R^+ e$ or $v \in \R^+ e$. In this case, $f(r_u u) = f(r_v v)$ is equivalent to the condition
\begin{equation}\label{eq:ratiocond}
	\frac{r_u}{r_v} = \frac{f(v)}{f(u)}.
\end{equation}
Moreover, we would like to express $e$ as a convex combination
\begin{equation}\label{eq:recover-e}
	\mu r_u u + (1-\mu) r_v v = e = \lambda u + (1- \lambda) v.
\end{equation}
Equating the coefficients in $u$ and $v$ in the left and right hand sides of \eqref{eq:recover-e} then brings us to the requirements
\begin{gather}\label{eq:mudef}
    \mu = \frac{\lambda}{r_u}, \, \text{ and }\, 1 - \mu = \frac{1-\lambda}{r_v}.
\end{gather}
From these and \eqref{eq:ratiocond}, we see that we must have
\begin{equation}
	\frac{\mu}{1 - \mu} = \frac{\lambda}{1-\lambda}\frac{r_v}{r_u} = \frac{\lambda}{1-\lambda}\frac{f(u)}{f(v)} \in \R^+,
\end{equation}
but this determines a single solution $\mu \in (0,1)$ from $u,v$ and $\lambda$ alone. With it, we can go back to \eqref{eq:mudef} to solve
\begin{equation}
	r_u = \frac{\lambda}{\mu} \,\text{ and }\, r_v = \frac{1- \lambda}{1 - \mu}.
\end{equation}
This finally brings us back to the first case, and we obtain
\begin{equation}\begin{aligned}
	f(e) &< \mu f(r_u u) + (1-\mu) f(r_v v) \\
	     &= \mu r_u f(u) + (1-\mu) r_v f(v) \\
	     &= \lambda f(u) + (1-\lambda) f(v),
\end{aligned}\end{equation}
which finishes the proof of \eqref{eq:gaugestrictineq}.
\end{proof}

The above lemma immediately implies:

\begin{lemma}\label{lem:sumextremals}
Let $X$ be a Banach space, and $f,g$ be convex positively one-homogeneous functionals for which $f(u)=g(u)=0$ only for $u=0$, and for which there exist index sets $\Phi, \Theta$, points $u_\varphi, u_\theta \in X$ for $\varphi \in \Phi$ and $\theta \in \Theta$ such that:
\begin{equation}
    \Ext \{u \, \mid \, f(u) \leq 1 \} = \left\{ \frac{u_\varphi}{f(u_\varphi)} \,\middle\vert\, \varphi \in \Phi\right\}\ \text{ and }\ \Ext \{u \, \mid \, g(u) \leq 1 \} = \left\{ \frac{u_\theta}{g(u_\theta)} \,\middle\vert\, \theta \in \Theta\right\}.
\end{equation}
Then we have
\begin{equation}
     \left\{ \frac{1}{f(u_\varphi) + g(u_\varphi)} u_\varphi \,\middle\vert\, \varphi \in \Phi\right\} \cup \left\{ \frac{1}{f(u_\theta) + g(u_\theta)} u_\theta \,\middle\vert\, \theta \in \Theta\right\} \subset \Ext \{u \, \mid \, f(u)+g(u) \leq 1 \}.
\end{equation}
\end{lemma}
\begin{proof}
We apply Lemma \ref{lem:gaugeextremals}, noticing that as soon as the inequality in \eqref{eq:gaugestrictineq} is strict for either $f$ or $g$, it is strict for the sum $f+g$ as well. 
\end{proof}

In fact, using Lemma \ref{lem:gaugeextremals} we also (quite surprisingly) get the case $p=1$ in the following result:

\begin{thm}\label{thm:masspenalizedext}
Let $p \in (0,1]$ and $\beta >0$. Then the extremal points of the set
\begin{equation}
	\left\{ \mu \in \M(\Omega)\, \middle \vert\, \mu(\Omega)=0, \  W_p(\mu^+, \mu^-)  + \frac{\beta}{2}|\mu|(\Omega) \leq 1 \right\}
\end{equation}
are all rescaled dipoles $
	\nu = \mathcal{D}_\beta(x,y)$ for  $(x,y) \in \Omega \times \Omega$ with $x \neq y$.
\end{thm}
\begin{proof}
Using Lemmas \ref{lem:balancedtvext} and \ref{lem:sumextremals} we get that all adequately normalized dipoles must be extremal. To see that there can be no other extremal points, we apply Proposition \ref{prop:onlydipolesext}.
\end{proof}

\begin{rem}
Notice that Proposition \ref{prop:dipolesareext}, which holds only for $p<1$ and is in fact false for $p=1$ by Proposition \ref{prop:lackofextremals}, does not play a role in the argument above.
\end{rem}

\subsection{Proof of Theorem \ref{thm:KRextremals}}
Now we aim to find the extremal points of the $\KR^\ab_p$ ball $\big\{ \mu\, \big \vert\, \|\mu\|_{\KR^\ab_p} \leq 1 \big\}$. This norm is expressed in \eqref{eq:def-KRbp} as an infimal convolution of positively one-homogeneous functionals, that is
\begin{equation}
	\|\mu\|_{\KR^\ab_p} = \inf_\nu\, \W^\beta_p(\nu) + \alpha|\mu - \nu|(\Omega),\, \text{ for }\, \W^\beta_p(\nu):=W_p(\nu^+, \nu^-) + \frac{\beta}{2}|\nu|(\Omega) + \textbf{1}_{\nu(\Omega) = 0}.
\end{equation}
In this situation, we have by \cite[Lem.~3.4]{IglWal22}, Theorem \ref{thm:masspenalizedext} and the characterization of extremal points for the total variation of measures that these can only be Dirac masses or rescaled dipoles. Our task is then to find out which of these are actually extremal to arrive at the characterization \eqref{eq:KRextremals}, in which the condition on the dipoles ensures that their transportation cost is strictly lower than their total variation. We begin by checking that the Dirac masses are indeed extremal, for which our proof follows the structure of that of \cite[Prop.~3.8]{IglWal22}.
\begin{prop}\label{prop:deltasextkr}All Dirac masses $\pm \delta_x / \alpha$ for $x \in \Omega$ are extremal points of $\big\{ \mu \, \big \vert \|\mu\|_{\KR^\ab_p} \leq 1 \big\}$.
\end{prop}
\begin{proof}
Without loss of generality, let us consider $\mu=\delta_x/\alpha$. Now, we claim that the problem
\begin{equation}\label{eq:innerminprobdelta}
    \inf_\nu\, \W^\beta_p(\nu) + \alpha|\delta_x - \nu|(\Omega)	
\end{equation}
has $0$ as its unique minimizer. To see this, let~$\bar{\nu}$ denote an arbitrary minimizer of~\eqref{eq:innerminprobdelta}. Since~$\W^\beta_p(\bar\nu)$ is finite, there holds~$\bar\nu(\Omega)=0$ and, consequently,
\begin{align*}
    1=\alpha\big|(\delta_x/\alpha) (\Omega)\big|&=\alpha\big|(\delta_x/\alpha) (\Omega)-\bar\nu(\Omega)\big| \leq \alpha\big|\delta_x/\alpha-\bar\nu\big|(\Omega) \\&\leq \inf_\nu\,\W^\beta_p(\nu) + \alpha\big|\delta_x/\alpha - \nu\big|(\Omega) \leq \alpha\big|\delta_x/\alpha\big|(\Omega)=1
\end{align*}
where the final inequality follows from
\begin{equation*}
    \inf_\nu\, \W^\beta_p(\nu) + \alpha\big|\delta_x/\alpha - \nu\big|(\Omega) \leq \W^\beta_p(0) + \alpha\big|\delta_x/\alpha - 0\big|(\Omega)=\alpha|\delta_x/\alpha|(\Omega).
\end{equation*}
These observations imply $\|\delta_x/\alpha\|_{\KR^\ab_p} = 1$ as well as~$\W^\beta_p(\bar\nu)=0$. Since the latter is only satisfied for~$\bar\nu=0$, the claimed uniqueness of~$\bar{\nu}=0$ follows.


Now, assume that we could express $\delta_x/\alpha$ as a convex combination $\delta_x/\alpha = \lambda \mu_1 + (1-\lambda)\mu_2$ with $\|\mu_i\|_{\KR^\ab_p} \leq 1$ and $\lambda \in (0,1)$. Now, denote by $\nu_i$ minimizers for the inner problem for $\|\mu_i\|_{\KR^\ab_p}$, so that
\begin{equation}\label{eq:innerminimizers}
    \alpha|\mu_i - \nu_i|(\Omega) \leq \alpha|\mu_i - \nu_i|(\Omega) + \W^\beta_p(\nu_i) = \|\mu_i\|_{\KR^\ab_p} = 1.
\end{equation}
Next, we take their convex combination $\nu_\lambda := \lambda \nu_1 + (1-\lambda)\nu_2$ and use it in \eqref{eq:innerminprobdelta}, which using convexity and \eqref{eq:innerminimizers} gives us
\begin{equation}
    1 = \|\delta_x/\alpha\|_{\KR^\ab_p} \leq \alpha|\delta_x/\alpha - \nu_\lambda|(\Omega) + \W^\beta_p(\nu_\lambda) \leq \lambda \|\mu_1\|_{\KR^\ab_p} + (1-\lambda)\|\mu_2\|_{\KR^\ab_p} = 1.
\end{equation}
But this means that $\nu_\lambda$ is a minimizer for \eqref{eq:innerminprobdelta}, so we must have $\lambda \nu_1 + (1-\lambda)\nu_2 = 0$, which implies
\begin{equation}
    \lambda(\mu_1 - \nu_1) + (1-\lambda)(\mu_2 - \nu_2) = \lambda \mu_1 + (1-\lambda)\mu_2 = \delta_x/\alpha.
\end{equation}
Remembering \eqref{eq:innerminimizers} we can now use the characterization of extremals for the total variation of measures, so that this convex combination must be trivial and $\mu_1 - \nu_1 = \mu_2 - \nu_2$. This means that
\begin{equation}
    \mu_i - \nu_i = \delta_x/\alpha, \text{ so }\alpha|\mu_i - \nu_i|(\Omega) = 1\text{ and }\W^\beta_p(\nu_i)=0 ,
\end{equation}
which in turn implies $\nu_i = 0$ and $\mu_i = \delta_x/\alpha$.
\end{proof}

\begin{prop}\label{prop:dipolesextkr}All rescaled dipoles with distance less than $2\alpha-\beta$, that is, elements of
\begin{equation}
    \left\{ \frac{1}{\beta +|x-y|^p}\big(\delta_x-\delta_y\big) \, \middle \vert\, x,y \in \Omega,\ 0 < |x-y|^p < 2\alpha-\beta \right\}
\end{equation}
are extremal points of $\big\{ \mu\, \big \vert\, \|\mu\|_{\KR^\ab_p} \leq 1 \big\}$.
\end{prop}
\begin{proof}
We would like to follow the same strategy as in Proposition \ref{prop:deltasextkr}, which requires that the inner minimization problem
\begin{equation}\label{eq:innerminprobdipole}
    \inf_\nu \,J_I(\nu)\ \text{ for }\ J_I(\nu) := \W^\beta_p(\nu) + \alpha\left|\mathcal{D}_\beta(x,y) - \nu\right|(\Omega)
\end{equation}
has the dipole $\mathcal{D}_\beta(x,y) = (\delta_x - \delta_y)/(\beta+|x-y|^p)$ as unique minimizer when $p \leq 1$ and $|x-y|^p < 2\alpha-\beta$.


Without loss of generality we can assume that $x=0$. First, we aim to reduce the problem to be supported on the segment $[0,y] = \{ \lambda y \, \mid \, \lambda \in [0,1]\}$. To do this, we could think of pushing forward any candidate to be a minimizer $\nu$ through the projection onto $[0,y]$, which is a convex set. This is not enough however, since even though this cannot increase transport cost or total variation, it could be that the fidelity term $|\mathcal{D}_\beta(0,y) - \nu|$ increases by this transformation. The cause of this is the lack of injectivity of the projection, which could cause cancellations of mass on $\{0\}$ or $\{y\}$. To avoid this pitfall we can define a transformation $\Xi$ by
\begin{equation}\label{eq:semiprojection}
    z \mapsto \Xi(z) := \begin{cases}
             \frac12 z &\text{ if } z \cdot \frac{y}{|y|} \leq 0,\\
             \big(z \cdot \frac{y}{|y|}\big) \frac{y}{|y|} + \frac12 \big(z - \big( z \cdot \frac{y}{|y|}\big) \frac{y}{|y|} \big) &\text{ if } 0 < z \cdot \frac{y}{|y|} < |y|,\\
             y + \frac12 (z-y) &\text{ if } z \cdot \frac{y}{|y|} \geq |y|.
             \end{cases}
\end{equation}
We notice that $\Xi^{-1}(\{0\}) = \{0\}$ and $\Xi^{-1}(\{y\}) = \{y\}$, so for the pushforward $\Xi_\# \nu$ we must have
\begin{equation}
    \left|\mathcal{D}_\beta(0,y) - \Xi_\# \nu\right| = \left|\mathcal{D}_\beta(0,y) - \nu\right|,
\end{equation}
while $\W^\beta_p(\Xi_\# \nu) \leq \W^\beta_p(\nu)$.  To see this, notice that $\W_p((\Xi_\# \nu)^+, (\Xi_\# \nu)^-) \leq \W_p(\nu^+, \nu^-)$ because $\Xi$ is $1$-Lipschitz, the pushforward does not increase the total variation of a measure, and $(\Xi_\# \nu)(\Xi(\Omega))=\nu(\Omega)=0$. 
Moreover, if the support of $\nu$ intersects the first or the third region in the definition \eqref{eq:semiprojection}, then $\W^\beta_p(\Xi_\# \nu) < \W^\beta_p(\nu)$. Indeed, an easy computation shows that $|\Xi(p_1) - \Xi(p_2)| < |p_1 - p_2|$ if $p_1$ or $p_2$ belong to the first or the third region in \eqref{eq:semiprojection}. Thus, if the support of $\nu$ intersects the first or the third region in \eqref{eq:semiprojection}, the strict inequality $\W_p(\Xi_\# \nu^+, \Xi_\# \nu^-) < \W_p(\nu^+, \nu^-)$ holds.
Further, we claim that if we had mass on the second region in \eqref{eq:semiprojection} but outside of the segment $[0,y]$, that is
\begin{equation}\label{eq:chargeslab}
    |\nu|\big( A \big) > 0 \ \text{ for }\ A:= \big\{z \in \Omega \,\big\vert\, 0 < z \cdot (y / |y|) < |y| \big\} \setminus [0,y]
\end{equation}
then also $J_I(\Xi_\# \nu) < J_I(\nu)$. To see this, first notice that we must have $\nu(\{0\})>0$ or $\nu(\{y\})<0$, since if $\nu(\{0\})\leq 0 \leq \nu(\{y\})$ we would have $J_I(0) < J_I(\nu)$, but $|y|^p < 2\alpha-\beta$ ensures that
\begin{equation}
  J_I(\nu) \leq J_I\left( \frac{1}{\beta+|y|^p}(\delta_0 - \delta_y) \right) < J_I(0),
\end{equation}
using the minimality on $\nu$.
Then, \eqref{eq:chargeslab} combined with $\nu(\{0\}) > 0$ or $\nu(\{y\}) < 0$ also implies that any optimal transportation plan $\gamma_{\opt}$ for $W_p(\nu^+, \nu^-)$ must satisfy
\begin{equation}\label{eq:diagonaltransport}
    \gamma_{\opt}\big( \{0\} \times A \big) >0\ \text{ or }\ \gamma_{\opt}\big( A \times \{y\}\big) >0.
\end{equation}
To see this, let us define
\begin{equation}
    \nu_A = \nu - (\pi_1)_\#\big[\gamma_{\opt}\mres \big( A \times A \big)\big] + (\pi_2)_\#\big[\gamma_{\opt}\mres\big( A \times A \big)\big]
\end{equation}
which, taking into account that $(\pi_1)_\# \gamma_{\opt} = \nu^+$, $(\pi_2)_\# \gamma_{\opt} = \nu^-$ and $\nu^+ \perp \nu^-$, remains balanced so that $\nu_A(\Omega)=0$. And if \eqref{eq:diagonaltransport} did not hold, we would have $\W^\beta_p(\nu_A) < \W^\beta_p(\nu)$ because 
\begin{equation}
    \gamma_{\opt} \mres \big( (\Omega \times \Omega) \setminus A \times A\big) \in \Pi(\nu_A^+,\nu_A^-) \ \text{ and }\ \beta|\nu_A|(\Omega) < \beta|\nu|(\Omega).
\end{equation}
Finally, noticing that for every $(p_1, p_2) \in \{0\} \times A$ or $(p_1, p_2) \in  A\times \{y\}$ we have $|\Xi(p_1) - \Xi(p_2)| < |p_1 - p_2|$, we obtain that
\begin{equation}\begin{aligned}
    W_p\big((\Xi_\# \nu)^+, (\Xi_\# \nu)^-\big) &\leq \int_{\Xi(\Omega) \times \Xi(\Omega)} |z-w|^p \dd(\Xi, \Xi)_\#\gamma_{\opt}(z,w) \\
    &< \int_{\Omega \times \Omega} |z-w|^p \dd\gamma_{\opt}(z,w)= W_p(\nu^+, \nu^-).
\end{aligned}\end{equation}
In conclusion, we must have $\supp \nu \subseteq [0,y]$ for any minimizer $\nu$ of \eqref{eq:innerminprobdipole}. Now, 
consider the following estimate for every balanced $\nu$ different from $\mathcal{D}_\beta(y,0) = \frac{\delta_0 - \delta_y}{\beta+|y|^p}$:
\begin{equation}\begin{aligned}\label{eq:esttria}
 J_I(\mathcal{D}_\beta(y,0)) &= W_p\left(\mathcal{D}_\beta(y,0)^+, \mathcal{D}_\beta(y,0)^- \right) + \frac{\beta}{2}\left| \mathcal{D}_\beta(y,0) \right|(\Omega) \\
& \leq W_p\left(\left(\mathcal{D}_\beta(y,0) - \nu\right)^+, \left(\mathcal{D}_\beta(y,0) - \nu\right)^-\right) + W_p(\nu^+, \nu^-)  + \frac{\beta}{2}|\mathcal{D}_\beta(y,0)|(\Omega)\\
& \leq  W_p(\nu^+, \nu^-) +  \frac{|y|^p}{2}  \left|\mathcal{D}_\beta(y,0) - \nu\right|(\Omega) +  \frac{\beta}{2} |\mathcal{D}_\beta(y,0)|(\Omega) \\
& <  W_p(\nu^+, \nu^-) +  \alpha |\mathcal{D}_\beta(y,0) - \nu|(\Omega) - \frac{\beta}{2}  |\mathcal{D}_\beta(y,0) - \nu|(\Omega)\nonumber + \frac{\beta}{2}  |\mathcal{D}_\beta(y,0)|(\Omega) \\
& \leq W_p(\nu^+, \nu^-) +  \alpha |\mathcal{D}_\beta(y,0) - \nu|(\Omega) + \frac{\beta}{2}|\nu|(\Omega) = J_I(\nu),
\end{aligned}\end{equation}
where in the first inequality we used the triangle inequality \eqref{eq:KRtriangle}, in the second that $\supp \nu \subseteq [0,y]$, and in the third the assumed bound $|y|^p < 2\alpha-\beta$. In particular, the dipole $\mathcal{D}_\beta(y,0)$ is the unique minimizer of $J_I$.\end{proof}

\begin{lemma}
Rescaled dipoles $\mathcal{D}_\beta(x,y)$ for which $|x-y|^p \geq 2\alpha - \beta$ are not extremal in $\big\{ \mu\, \big \vert\, \|\mu\|_{\KR^\ab_p} \leq 1 \big\}$.
\end{lemma}
\begin{proof}
We just notice that these can be decomposed as
\begin{equation}
    \frac12 \left(\frac{2}{\beta+|x-y|^p}\delta_x\right) + \frac12 \left(-\frac{2}{\beta+|x-y|^p}\delta_y\right), 
\end{equation}
and that $\|c\delta_x\|_{\KR^\ab_p} = \|c\delta_y\|_{\KR^\ab_p} = \alpha c$ for all $c>0$.
\end{proof}

\section{Minimization problems with \texorpdfstring{$\KR_p^\ab$}{KR} regularization}\label{sec:minimizationproblem}
In this section we consider the following variational problem
\begin{equation} \label{def:problem}
    \inf_{\mu \in \mathcal{M}(\Omega)} J(\mu) \coloneqq F(K\mu)+\|\mu\|_{\KR_p^\ab} \tag{$\mathcal{P}$},
\end{equation}
where $Y$ is a given Hilbert space, the forward operator $K : \mathcal{M}(\Omega) \rightarrow Y$ is  weak*-to-strong continuous, and the discrepancy $F\colon Y \to \R$ is strictly convex, Frechet differentiable and $\nabla F$ is Lipschitz continuous on compact sets.
Moreover, from here on we assume that the set $\Omega$ is compact.

\subsection{Existence of minimizers}

Under the given assumptions on \eqref{def:problem}, the existence of solutions is a straightforward application of the direct method of calculus of variations.
\begin{thm}
There exists at least one solution~$\bar{\mu}$ of \eqref{def:problem}.
\end{thm}
\begin{proof}
Since $F$ is bounded from below, the infimum in \eqref{def:problem} is finite. Moreover, since $\beta >0$, the functional $J$ is coercive in $\mathcal{M}(\Omega)$ due to the trivial bound \eqref{eq:coercivity}. Using Lemma \ref{lem:lscKR} and the assumptions on $K$ and $F$ we also infer that $J$ is weak* lower semicontinuous. Therefore existence of minimizers for \eqref{def:problem} follows by a straightforward application of the direct method of calculus of variations.
\end{proof}

\subsection{First-order optimality conditions}
This section is devoted to the derivation of first-order necessary and sufficient optimality conditions for Problem~\eqref{def:problem}. We obtain the following characterization. 
\begin{thm} \label{thm:firstorder}
Let~$\bar{\mu} \in \mathcal{M}(\Omega)$ be given. Moreover, let~$\bar{\nu}\in \mathcal{M}(\Omega)$,~$\bar{\nu}(\Omega)=0$, and~$\bar{\gamma} \in \Pi(\bar{\nu}_+, \bar{\nu}_-)$ be such that
\[
 \|\bar\mu\|_{\KR_p^\ab} =  \mathcal W^\beta_p(\bar \nu)  + \alpha|\bar\mu - \bar\nu|(\Omega)
\] 
as well as
\[
 W_p(\bar \nu)= \int_{\Omega \times \Omega} |x-y|^p~\dd\bar{\gamma}(x,y).
\]
Finally set~$\bar{q}:=-K_* \nabla F(K\bar\mu)$ as well as
\[
\Psi_{\bar{q}}(x,y):= \frac{\bar{q}(x)-\bar{q}(y)}{|x-y|^p+\beta} \quad \text{for all}~(x,y) \in\Omega\times\Omega.
\]
The following statements are equivalent:
\begin{itemize}
\item[1.] The measure~$\bar{\mu}$ is a solution to~\eqref{def:problem}. \label{item1}
\item[2.] There holds~$\bar{q} \in  \partial \|\bar\mu\|_{\KR_p^\ab}  $. \label{item2}
\item[3.] \label{item3} There holds~$|\bar{q}(x)|\leq \alpha$,~$\Psi_{\bar{q}}(x,y)\leq 1$ for all~$x,y \in \Omega$ as well as
\begin{equation}
 \supp(\bar{\mu}-\bar{\nu})^{\pm} \subset \left\{\,x \in \Omega\;|\;\bar{q}(x)=\pm \alpha\,\right\},\quad \supp \bar{\gamma} \subset \left\{\,(x,y) \in \Omega \times \Omega\;|\;\Psi_{\bar{q}}(x,y)= 1\,\right\}.
\end{equation}
\end{itemize}                                                                               
\end{thm}
The proof is split into several parts. First, we note that $\partial \|\bar\mu\|_{\KR_p^\ab}$ is given by the intersection of the subdifferentials of the two terms in it:

\begin{prop}\label{prop:subdiffKR}
We have 
\begin{equation}
\bar{q} \in \partial \|\bar\mu\|_{\KR_p^\ab} \Leftrightarrow \bar{q} \in \alpha \partial \|\bar\mu-\bar\nu\|_\mathcal{M} \cap \partial \mathcal{W}^\beta_p(\bar{\nu}).
\end{equation}
\end{prop}
\begin{proof}
Since the infimal convolution in $\|\bar\mu\|_{\KR_p^\ab}$ is exact, i.e.,
\[
\|\bar\mu\|_{\KR_p^\ab} = \inf_{\nu } \left \lbrack \mathcal W^\beta_p( \nu)  + \alpha|\bar\mu - \nu|(\Omega)\right \rbrack= \mathcal W^\beta_p(\bar \nu)  + \alpha|\bar\mu - \bar\nu|(\Omega),
\]
the claimed statement follows from~\cite[Cor.~2.4.7]{Zal02}.
\end{proof}
%
As a consequence, it suffices to characterize the sets~$\alpha\partial \|\bar\mu-\bar\nu\|_\mathcal{M}$ and~$\partial \mathcal{W}^\beta_p(\bar{\nu})$, respectively. In both cases, we make use of the following auxiliary result which is based on~\cite[Lemma~3.1]{linconv} as well as the characterization of the extremal points of~$\{\,\mu \in \mathcal{M}(\Omega)\;|\;|\mu|(\Omega)\leq 1\,\}$ and~$\{\,\nu \in \mathcal{M}(\Omega)\;|\;\mathcal{W}^\beta_p({\nu})\leq 1\,\}$, respectively.
\begin{lemma} \label{lem:linearoverext}
Let~$\bar{q} \in \mathcal{C}(\Omega)$ be given. Then there holds
\[
\max_{|\mu|(\Omega)\leq 1} \langle \bar{q}, \mu \rangle= \max_{x \in \Omega} |\bar{q}(x)|,~\max_{\mathcal{W}^\beta_p({\nu})\leq 1} \langle \bar{q}, \nu \rangle= \max_{(x,y) \in \Omega \times \Omega} \Psi_{\bar{q}}(x,y).
\]
\end{lemma}
\begin{proof}
From~\cite[Lemma 3.1]{linconv} and \cite[Lemma 3.10]{linconv}, respectively, we get that
\[
\max_{|\mu|(\Omega)\leq 1} \langle \bar{q}, \mu \rangle=\max_{x \in \Omega, \sigma \in \{-1,1\} } \sigma\langle \bar{q}, \delta_x \rangle= \max_{x \in \Omega} |\bar{q}(x)|.
\]
Similarly, now invoking Theorem~\ref{thm:masspenalizedext} (and again using \cite[Lemma 3.1]{linconv}), there holds
\[
\max_{\mathcal{W}^\beta_p({\nu})\leq 1} \langle \bar{q}, \nu \rangle= \max_{(x,y)\in \Omega \times \Omega} \frac{\bar{q}(x)-\bar{q}(y)}{|x-y|^p+\beta}=\max_{(x,y)\in \Omega \times \Omega} \Psi_{\bar{q}}(x,y).\qedhere
\]
\end{proof}

\begin{lemma} \label{lem:subdiffM}
The function~$\bar{q} \in \mathcal{C}(\Omega)$ satisfies~$\bar{q} \in \alpha \partial |\bar\mu-\bar\nu|(\Omega)$ if and only if
\begin{equation} \label{eq:suppcondM}
|\bar{q}(x)|\leq \alpha, \quad \supp(\bar{\mu}-\bar{\nu})^{\pm} \subset \left\{\,x \in \Omega\;|\;\bar{q}(x)=\pm \alpha\,\right\}.
\end{equation}
\end{lemma}
\begin{proof}
Since~$|\cdot|(\Omega)$ is positively one-homogeneous, there holds~$\bar{q} \in \alpha\partial |\bar\mu-\bar\nu|(\Omega)$ if and only if
\[
\max_{|\mu|(\Omega)\leq 1} \langle \bar{q}, \mu \rangle \leq \alpha, ~ \langle \bar{q}, \bar\mu-\bar\nu  \rangle= \alpha|\bar\mu-\bar\nu|(\Omega) 
\]
According to Lemma~\ref{lem:linearoverext} this is equivalent to
\begin{equation} \label{eq:suppcondMhelp}
\max_{x\in\Omega} | \bar{q}(x)| \leq \alpha, ~ \langle \bar{q}, \bar\mu-\bar\nu  \rangle= \alpha |\bar\mu-\bar\nu|(\Omega) 
\end{equation}
Finally,~\cite[Lemma~3.4]{KK} yields the equivalence of~\eqref{eq:suppcondMhelp} and~\eqref{eq:suppcondM}.
\end{proof}
\begin{lemma} \label{lem:subdiffT}
The function~$\bar{q} \in \mathcal{C}(\Omega)$ satisfies~$\bar{q} \in \partial\mathcal{W}^\beta_p(\bar{\nu}) $ if and only if
\[
\Psi_{\bar{q}}(x,y)\leq 1 \quad\text{for all}~(x,y) \in \Omega \times \Omega, \quad\text{and}~ \ \supp \bar{\gamma} \subset \left\{\,(x,y) \in \Omega \times \Omega\;|\;\Psi_{\bar{q}}(x,y)= 1\,\right\}. 
\]
\end{lemma}
\begin{proof}
First, since~$\bar{\gamma} \in \Pi(\bar{\nu}^+,\bar{\nu}^-)$, we have
\[
|\bar{\gamma}|(\Omega\times \Omega)= \int_{\Omega \times \Omega}1~\mathrm{d}\bar{\gamma}(x,y)= \int_{\Omega} 1~\mathrm{d}\bar{\nu}^+= |\bar{\nu}^+|(\Omega)= \frac{1}{2} |\bar{\nu}|(\Omega) 
\]
and thus
\begin{equation} \label{eq:altrepofW}
    \mathcal{W}^\beta_p(\bar{\nu})= \int_{\Omega \times \Omega} |x-y|^p+\beta ~\mathrm{d}\bar{\gamma}(x,y).
\end{equation}
Second, as in Lemma~\ref{lem:subdiffM}, there holds~$\bar{q} \in \partial\mathcal{W}^\beta_p(\bar{\nu}) $ if and only if
\begin{equation} \label{eq:helpsubdiffM}
\max_{\mathcal{W}^\beta_p({\nu})\leq 1 } \langle \bar{q}, \nu\rangle \leq 1, \quad \langle \bar{q},  \bar{\nu}\rangle= \mathcal{W}^\beta_p(\bar{\nu})
\end{equation}
due to the positive one-homogeneity of~$\mathcal{W}^\beta_p$. Invoking Lemma~\ref{lem:linearoverext},~\eqref{eq:helpsubdiffM} holds if and only if
\[
\max_{(x,y) \in \Omega\times\Omega } \Psi_{\bar{q}}(x,y) \leq 1, \quad \langle \bar{q},  \bar{\nu}\rangle= \mathcal{W}^\beta_p(\bar{\nu}).
\]
Now, again using $\bar{\gamma} \in \Pi(\bar{\nu}^+,\bar{\nu}^-)$ as well as the definition of~$\Psi_{\bar{q}}$, we get
\[
    \langle \bar{q},  \bar{\nu}\rangle= \langle \bar{q},\bar{\nu}^+ \rangle- \langle \bar{q},\bar{\nu}^- \rangle= \int_{\Omega\times\Omega} \bar{q}(x)-\bar{q}(y)~\mathrm{d}\bar{\gamma}(x,y)=\int_{\Omega\times\Omega} \Psi_{\bar{q}}(x,y)\left(|x-y|^p +\beta\right)~\mathrm{d}\bar{\gamma}(x,y).
\]
Consequently, see also~\eqref{eq:altrepofW},~$\langle \bar{q},  \bar{\nu}\rangle= \mathcal{W}^\beta_p(\bar{\nu})$ is equivalent to
\begin{equation}\label{eq:presupphamma}
    \int_{\Omega\times\Omega}\left( \Psi_{\bar{q}} (x,y) -1 \right)\left(|x-y|^p+\beta \right)~\mathrm{d}\bar{\gamma}(x,y)=0.
\end{equation}
Since~$\Psi(x,y)\leq 1$ for all~$(x,y)\in\Omega \times \Omega$ the integrand above is non-positive. Hence, due to the positivity of~$\bar{\gamma}$,~\eqref{eq:presupphamma} holds if and only if
\[
   \supp \bar{\gamma} \subset \left\{\,(x,y) \in \Omega \times \Omega\;|\;\Psi_{\bar{q}}(x,y)= 1\,\right\}\qedhere
\]
\end{proof}
Combining the observations of Lemma~\ref{lem:subdiffM} and~\eqref{lem:subdiffT}, respectively, with Proposition~\ref{prop:subdiffKR}, we are finally able to prove Theorem~\ref{thm:firstorder}.
\begin{proof}
[Proof of Theorem~\ref{thm:firstorder}]
Since~the objective functional~$J$ in~\eqref{def:problem} is convex,~$\bar{\mu}\in \mathcal{M}(\Omega)$ is a solution to~\eqref{def:problem} if and only if~$0 \in \partial J(\bar{\mu})$. 
Note that the function~$f= F \circ K$ is convex, weak*-to-strong continuous and G\^ateaux differentiable. Its G\^ateaux  derivative at~$\mu \in \mathcal{M}(\Omega)$ in the direction of~$\delta \mu \in \mathcal{M}(\Omega)$ is given by
\[
f'(\mu) \delta \mu= \langle K_* \nabla F(K\bar{u}), \delta \mu  \rangle=-\langle \bar{q}, \delta \mu  \rangle.
\]
Consequently, due to the characterization of the subgradient for G\^ateaux-differentiable functions as well as the sum rule (see \cite[Prop. 5.3, Prop 5.6]{temam}, for example) we arrive at
\[
0 \in \partial J(\bar{\mu}) \Leftrightarrow \bar{q} \in  \partial \|\bar\mu\|_{\KR_p^\ab}.
\]
This proves~$(1.)\Leftrightarrow (2.)$ in Theorem~\ref{thm:firstorder}.

Thus, it remains to show~$(2.)\Leftrightarrow (3.)$. However, this immediately follows from Proposition~\ref{prop:subdiffKR} taking into account Lemma~\ref{lem:subdiffM} and Lemma~\ref{lem:subdiffT}, respectively.
\end{proof}
\begin{prop} \label{e}
Let~$\bar{\mu}$ be a solution to~\eqref{def:problem} and let~$\bar{\nu},~\bar{\gamma}$ as well as~$\bar{q}$ and~$\Psi_{\bar{q}}$ be as in Theorem~\ref{thm:firstorder}. Moreover, assume that there are~$\bar{N}_1,~\bar{N}_2 \in \mathbb{N} $ as well as finite sets~$\left\{\bar{z}_i\right\}^{\bar{N}_1}_{i=1}\subset \Omega$ and~$\left\{(\bar{x}_j,\bar{y}_j)\right\}^{\bar{N}_2}_{j=1}\subset \Omega\times \Omega$, respectively, with
\[
  \left\{\,x \in \Omega\;|\;\bar{q}(x)=\pm \alpha\,\right\}=\left\{\bar{z}_i\right\}^{\bar{N}_1}_{i=1},\quad \left\{\,(x,y) \in \Omega \times \Omega\;|\;\Psi_{\bar{q}}(x,y)= 1\,\right\}=\left\{(\bar{x}_j,\bar{y}_j)\right\}^{\bar{N}_2}_{j=1}.
\]
Then there are coefficients~$\bar{\zeta}_i, \bar{\lambda}_j \geq 0$,~$i=1,\dots,\bar{N}_1$,~$j=1,\dots,\bar{N}_2$, such that
\begin{equation} \label{eq:sparsesolstruc}    \bar{\mu}=\sum^{\bar{N}_1}_{i=1} (\bar{q}(\bar{z}_i)/\alpha)\bar\zeta_i \delta_{\bar{z}_i}+\sum^{\bar{N}_2}_{j=1} \bar\lambda_j \mathcal{D}_\beta(\bar{x}_j, \bar{y}_j) , \quad \bar{\nu}=\sum^{\bar{N}_2}_{j=1}\bar\lambda_j \mathcal{D}_\beta(\bar{x}_j, \bar{y}_j).
\end{equation}
Moreover there holds
\[
    \bar{\gamma}= \sum^{\bar{N}_2}_{j=1}\bar\lambda_j\frac{\delta_{(\bar x_j,\bar y_j)}}{|{\bar{x}_j}-{\bar{y}_j}|^p+\beta} ,\quad |\bar{\mu}-\bar{\nu}|(\Omega)=\sum^{\bar{N}_1}_{i=1} \bar \zeta_i, \quad \mathcal{W}^\beta_p(\bar \nu)=\sum^{\bar{N}_2}_{j=1} \bar \lambda_j.
\]
\end{prop}
\begin{proof}
By assumption and Theorem~\ref{thm:firstorder}, we have~$\supp \bar{\gamma}\subset\left\{(\bar{x}_j,\bar{y}_j)\right\}^{\bar{N}_2}_{j=1}$, i.e., there are coefficients~$\bar{\lambda}_j \geq 0$,~$j=1,\dots,\bar{N}_2$, with
\[
   \bar{\gamma}= \sum^{\bar{N}_2}_{j=1}\bar\lambda_j\frac{\delta_{(\bar x_j,\bar y_j)}}{|{\bar{x}_j}-{\bar{y}_j}|^p+\beta} \quad \text{and thus} \ ~\bar{\nu}=\sum^{\bar{N}_2}_{j=1}\bar\lambda_j \mathcal{D}_\beta(\bar{x}_j, \bar{y}_j)
\]
since~$\bar{\gamma}\in\Pi(\bar{\nu}^+,\bar{\nu}^-)$. Now, again invoking Theorem~\ref{thm:firstorder} as well as~$\supp(\bar{\mu}-\bar{\nu})^{\pm} \subset \left\{\bar{z}_i\right\}^{\bar{N}_1}_{i=1}$, yields coefficients~$\bar{\zeta}_i \geq 0$,~$i=1,\dots,\bar{N}_1$, with
\begin{equation} \label{eq:repofdiff}
   \alpha( \bar{\mu}-\bar{\nu} )= \sum^{\bar{N}_1}_{i=1} \bar{q}(\bar{z}_i)\bar\zeta_i \delta_{\bar{z}_i} \quad \text{which implies} \ ~|\bar{\mu}-\bar{\nu}|(\Omega)=\sum^{\bar{N}_1}_{i=1} \bar \zeta_i
\end{equation}
as well as~\eqref{eq:sparsesolstruc}.
Finally, see the proof of Theorem~\ref{thm:firstorder}, recall that~$\langle \bar{q},\bar{\nu}\rangle=\mathcal{W}^\beta_p(\bar{\nu})$ as well as
\[
    \mathcal{W}^\beta_p \left(\nu \right)=1 \quad \text{for all}~\nu= \frac{\delta_x-\delta_y}{|x-y|^p+\beta},~(x,y)\in \Omega \times \Omega.
\]
Consequently, we have
\[
\sum^{\bar{N}_2}_{j=1} \bar{\lambda}_j=\langle\bar{q},\bar{\nu}\rangle=\mathcal{W}^\beta_p(\bar{\nu}) \leq \sum^{\bar{N}_2}_{j=1} \bar{\lambda}_j
\]
where the inequality follows from the convexity and positive one-homogeneity of~$\mathcal{W}^\beta_p$.
\end{proof}


\subsection{Algorithmic solution}
\label{subsec:algo}
This section is devoted to describing the application of an~\textit{accelerated generalized conditional gradient method} (AGCG) to Problem~\eqref{def:problem}, for which we abbreviate \[B:=\{\mu \,|\, \|\mu\|_{\KR_p^{\alpha,\beta}}\leq 1\}.\] The AGCG algorithm for non-smooth minimization, see~\cite{linconv} for the abstract algorithm in general Banach spaces, relies on the characterization of the extremal points and alternates between the update of a finite set of extremal points~$\mathcal{A}_k$ as well as of an iterate~$\mu_k$ in~$\operatorname{cone}(\mathcal{A}_k)$, the convex cone spanned by~$\mathcal{A}_k$. Complexity-wise, every iteration of AGCG requires the solution of two subproblems: The minimization of a linear functional over~$\Ext(B)$, to update~$\mathcal{A}_k$, as well as the solution of a finite-dimensional, constrained minimization problem to improve the iterate~$\mu_k$. While the latter can be done by standard methods, e.g. FISTA, interior point, or generalized Newton methods, we show that the former is equivalent to solving two finite-dimensional, non-convex minimization problems. Moreover, based on the abstract results in~\cite{linconv}, we present sufficient non-degeneracy conditions for the (fast) convergence of AGCG for~\eqref{def:problem}.

\subsubsection{Description of the AGCG method}
For a finite, ordered set of extremal points~$\mathcal{A}=\{\mu_j\}^N_{j=1}\subset \Ext(B)$ consider the finite-dimensional problem
\begin{equation} \label{def:subprop}
 \min_{\lambda \in \R_+^{N}} F\left(\sum_{j=1}^{N} \lambda_j K \mu_j\right) + \sum_{j=1}^{N} \lambda^j \tag{$\mathcal{P}(\mathcal{A})$} 
\end{equation}
where~$\R_+^{N}$ denotes the cone of componentwise non-negative vectors in~$\R^N$ and $N \in \N$. 
The AGCG method relies on the iterative update of the active set~$\mathcal{A}_k:=\{\mu^k_j\}^{N_k}_{j=1}$ as well as of an iterate~$\mu_k$ satisfying
\begin{equation} \label{eq:charofit}
 \mu_k=\sum^{N_k}_{j=1} \lambda^k_j \mu^k_j, \quad \lambda^k \in \argmin \eqref{def:subprop}, \quad \lambda^k_j >0,~j=1,\dots, N_k.
\end{equation}
Its $k$-th iteration can be described as follows. Given the current iterate~$\mu_k$ in the form~\eqref{eq:charofit}, we first compute~$
q_k \in C(\Omega) $ as well as a new candidate extremal point~$\widehat{\mu}_k \in \Ext(B)$ as defined by
\[
    q_k=-K_* \nabla F(K\mu_k),\quad \langle q_k, \widehat{\mu}_k \rangle=\max_{\mu \in \Ext(B)} \langle q_k, \mu \rangle.
\]
As defined in~\cite{linconv}, the algorithm stops with~$\mu_k=\bar{\mu}$ a minimizer to~\eqref{def:problem} if~$\langle q_k, \widehat{\mu}_k \rangle\leq 1$. Otherwise,~$\widehat{\mu}_k$ is added to the active set, i.e.,
\[
    \mathcal{A}^+_k= \mathcal{A}_k \cup \{\widehat{\mu}_k\}.
\]
Then, renaming~$\mathcal{A}^+_k=\{\mu^{k,+}_j\}^{N^+_k}_{j=1}$, we find the new iterate $\mu_{k+1}$ by solving~$(\mathcal{P}(\mathcal{A}^+_k))$ and setting
\[
    \mu_{k+1}=\sum^{N^+_k}_{j=1} \lambda^{k,+}_j \mu^{k,+}_j, \quad \lambda^{k,+} \in \argmin (\mathcal{P}(\mathcal{A}^+_k)).
\]
As a final step, unnecessary extremal points, i.e. those that are assigned a zero weight, are removed from~$\mathcal{A}^+_k$ by setting
\[
   \mathcal{A}_{k+1} :=  \mathcal{A}^+_k \setminus \{\mu_j^{k,+} : \lambda^{k,+}_j = 0\}.
\]
This ensures~\eqref{eq:charofit} for~$k=k+1$.

In the following, we address the computation of the new extremal point~$\widehat{\mu}_k$ which is required for checking the convergence of the method and the update of the active set~$\mathcal{A}_k$. Introducing the auxiliary variable
\[
    \Psi_{q_k}(x,y)= \frac{{q}_k (x)-{q}_k (y)}{|x-y|^p+\beta} \quad \text{for all}~(x,y) \in\Omega\times\Omega
\]
this can be done by computing a global extremum of~$q_k$ and a global maximum of~$\Psi_{q_k}$.
\begin{lemma} \label{lem:terminationcriterion}
Let~$\mu_k$ denote the current iterate of the AGCG method and define the dual variable~$q_k=-K_*\nabla F(K\mu_k)$. Moreover set
\[
    \Psi_{q_k}(x,y)= \frac{{q}_k (x)-{q}_k (y)}{|x-y|^p+\beta} \quad \text{for all}~(x,y) \in\Omega\times\Omega.
\]
Then there holds
\begin{equation} \label{eq:one}
   \max_{z\in\Omega} |q_k(z)|\leq \alpha,~\max_{(x,y)\in\Omega\times \Omega} \Psi_{q_k}(x,y)\leq 1  \ \Leftrightarrow \ \max_{\mu \in \Ext(B) } \langle q_k, \mu \rangle \leq 1.
\end{equation}
\end{lemma}
\begin{proof}
Note that
\begin{equation}\begin{aligned} \label{eq:estforlin}
    \max_{\mu \in \Ext(B) } \langle q_k, \mu \rangle &= \max \left\{\max_{z \in \Omega} \pm \langle q_k,\delta_z/\alpha \rangle,~ \sup_{\substack{(x,y)\in \Omega \times \Omega,\\|x-y|<2\alpha-\beta}} \langle q_k, \mathcal{D}_\beta(x,y)\rangle\right\} \\ &= \max \left\{\max_{z \in \Omega} \frac{|q_k(z)|}{\alpha},~ \sup_{\substack{(x,y)\in \Omega \times \Omega,\\|x-y|^p<2\alpha-\beta}} \Psi_{q_k}(x,y)\right\}
    \\ & \leq \max \left\{\max_{z \in \Omega} \frac{|q_k(z)|}{\alpha},~ \max_{(x,y)\in \Omega \times \Omega} \Psi_{q_k}(x,y)\right\}.
\end{aligned}\end{equation}
This immediately gives the~``$\Rightarrow$'' direction in~\eqref{eq:one}. Now assume that~$\max_{\mu \in \Ext(B) } \langle q_k, \mu \rangle \leq 1$. Then, due to the second equality in~\eqref{eq:estforlin}, we have
\[
    \max_{z\in\Omega} |q_k(z)|\leq \alpha,~\sup_{\substack{(x,y)\in \Omega \times \Omega,\\|x-y|^p<2\alpha-\beta}} \Psi_{q_k}(x,y)\leq 1 .
\]
Moreover, using~$\max_{z\in\Omega} |q_k(z)|\leq \alpha$, we conclude
\begin{align*}
    \sup_{\substack{(x,y)\in \Omega \times \Omega,\\|x-y|^p \geq 2\alpha-\beta}} \Psi_{q_k}(x,y) \leq \sup_{\substack{(x,y)\in \Omega \times \Omega,\\|x-y|^p \geq 2\alpha-\beta}} \frac{2\alpha}{\beta+|x-y|^p}\leq 1.
\end{align*}
Combining both observations, finishes the proof of the~``$\Leftarrow$'' direction.
\end{proof}
\begin{prop}
 \label{prop:sollinear}
 Let~$q_k$ and~$\Psi_{q_k}$ be defined as in~Lemma~\ref{lem:terminationcriterion} and assume that
 \[
 \max_{\mu\in\Ext(B)} \langle q_k,\mu \rangle >1.
 \]
 Moreover, let~$z_k\in \Omega$ and~$(x_k,y_k) \in \Omega \times \Omega$ be such that
 \[
      z_k \in \argmax_{z\in\Omega} |q_k(z)|,~(x_k,y_k) \in \argmax_{(x,y)\in\Omega\times \Omega} \Psi_{q_k}(x,y).
 \]
 Finally set
 \[
     \widehat{\mu}_k:= \sign(q_k(z_k))\frac{\delta_{z_k}}{\alpha} \quad \text{if} \quad \max_{z\in\Omega} |q_k(z)|\geq \max_{(x,y)\in\Omega\times \Omega} \Psi_{q_k}(x,y)
 \]
 and~$\widehat{\mu}_k:=\mathcal{D}_\beta(x_k,y_k)$, otherwise. Then there holds~$\widehat{\mu}_k \in \Ext(B)$ and
 \begin{equation}\label{def:linprob}
    \langle q_k,\widehat{\mu}_k \rangle= \max_{\mu\in\Ext(B)} \langle q_k,\mu \rangle.
 \end{equation}
\end{prop}
 \begin{proof}
 First note that~$q_k \neq 0$ by assumption. Then, using~\eqref{eq:estforlin}, we arrive at
 \[
 \max_{\mu \in \Ext(B) } \langle q_k, \mu \rangle = \max \left\{\max_{z \in \Omega} | q_k(z)|,~ \sup_{\substack{(x,y)\in \Omega \times \Omega,\\|x-y|^p<2\alpha-\beta}} \Psi_{q_k}(x,y)\right\}
 \]
 Hence, if~$\max_{z\in\Omega} |q_k(z)|\geq \max_{(x,y)\in\Omega\times \Omega} \Psi_{q_k}(x,y)$, there holds
 \[
     \max_{\mu\in\Ext(B)} \langle q_k,\mu \rangle= \max_{z\in\Omega} |q_k(z)|= \langle q_k, \sign(q_k(z_k))\delta_{z_k}\rangle
 \]
 and~$\widehat{\mu}_k= \sign(q_k(z_k))\delta_{z_k}/\alpha \in \Ext(B)$ satisfies~\eqref{def:linprob}.
 
 Now assume that
 \begin{equation}\label{eq:poss2}
     \max_{z\in\Omega} |q_k(z)|< \max_{(x,y)\in\Omega\times \Omega} \Psi_{q_k}(x,y).
 \end{equation}
 Again using~\eqref{eq:estforlin}, there holds
 \[
     \max_{\mu\in\Ext(B)} \langle q_k,\mu \rangle \leq \max_{(x,y)\in\Omega\times \Omega} \Psi_{q_k}(x,y)=\Psi_{q_k}(x_k,y_k)= \langle\Psi_{q_k}, \mathcal{D}_\beta(x_k,y_k) \rangle.
 \]
 Thus, to finish the proof, it suffices to show that~$\mathcal{D}_\beta(x_k,y_k)\in \Ext(B)$, i.e.,~$|x_k-y_k|^p< 2\alpha-\beta$. Due to~\eqref{eq:poss2}, we conclude
 \[
    (\beta+|x_k-y_k|^p) \max_{z\in\Omega} |q_k(z)| < q_k(x_k)-q_k(y_k) \leq 2 \max_{z\in\Omega} |q_k(z)|.
 \]
 Noting that~$\max_{z\in\Omega} |q_k(z)|>0$ finally yields the desired result,
 \end{proof}
Lemma \ref{lem:terminationcriterion} provides an explicit way to check the stopping criteria of the AGCG method by computing both $\max_{z\in\Omega} |q_k(z)|$ and $\max_{(x,y)\in\Omega\times \Omega} \Psi_{q_k}(x,y)$.
Proposition \ref{prop:sollinear} is instead allowing to compute the newly inserted extremal point at each iteration of the AGCG method. Indeed, if the stopping condition is not satisfied, i.e., $\max_{\mu\in\Ext(B)} \langle q_k,\mu \rangle >1$, then the newly inserted extremal point can be determined from $|q_k(\cdot)|$ and $\Psi_{q_k}$ as described in Proposition \ref{prop:sollinear}.

Following these considerations, the AGCG method described above is schematically summarized in Algorithm \ref{alg:accgcg}. 

\begin{algorithm}[t]
\begin{flushleft}
\hspace*{\algorithmicindent}\textbf{Input:}~$\mathcal{A}_0=\{\mu^0_j\}^{N_0}_{j=1}\subset \operatorname{Ext}(B)$.
\\
\hspace*{\algorithmicindent} \textbf{Output:} Minimizer~$\bar{\mu}$ to~\eqref{def:problem}.
\end{flushleft}
\begin{algorithmic}
\STATE 1. Find~$\lambda^0$ by solving~$(\mathcal{P}(\mathcal{A}_0))$ and update
\[
\mu_1:=\sum^{N_0}_{j=1} \lambda^0_j \mu^0_j,~\mathcal{A}_1:= \mathcal{A}_0 \setminus \left\{\,\mu^0_j\;|\;\lambda^0_j=0\,\right\}.
\]

\FOR {$k=1,2,\dots$}
\vspace*{0.3em}
\STATE 2. Given the current iterate~$\mu_k$ and active set~$\mathcal{A}_k=\{\mu^k_j\}^{N_k}_{j=1}$ compute
\[
    q_k=-K_*\nabla F(K\mu_k),\quad \Psi_{q_k}(x,y)= \frac{{q}_k (x)-{q}_k (y)}{|x-y|^p+\beta} \quad \text{for all}~(x,y) \in\Omega\times\Omega
\]
as well as~$z_k \in \Omega$ and~$(x_k,y_k) \in\Omega\times\Omega$ with
\[
    z_k \in \argmax_{z\in\Omega} |q_k(z)|,~(x_k,y_k) \in \argmax_{(x,y)\in\Omega\times \Omega} \Psi_{q_k}(x,y)
\]

\IF{$\max_{z\in\Omega} |q_k(z)|\leq \alpha,~\max_{(x,y)\in\Omega\times \Omega} \Psi_{q_k}(x,y)\leq 1$}
\vspace*{0.3em}
\STATE 3. Terminate with~$ \bar{\mu}=\mu_k$ a stationary point to~\eqref{def:problem}.
\vspace*{0.3em}
\ELSIF{$\max_{z\in\Omega} |q_k(z)|/\alpha \geq \max_{(x,y)\in\Omega\times \Omega} \Psi_{q_k}(x,y)$}
\vspace*{0.3em}
\STATE 4. Define~$\widehat{\mu}_k=\sign(q_k(z_k))\delta_{z_k}/\alpha$.
\ELSIF{$\max_{z\in\Omega} |q_k(z)|/\alpha < \max_{(x,y)\in\Omega\times \Omega} \Psi_{q_k}(x,y)$}
\STATE 5. Define~$\widehat{\mu}_k=\mathcal{D}_\beta(x_k,y_k)$.
\ENDIF

\vspace*{0.3em}
\STATE 6. Update the active set~$\mathcal{A}^+_k:=\mathcal{A}_k \cup \{\widehat{\mu}_k\}$.
\STATE 5. Denoting~$\mathcal{A}^+_k=\{\mu^{k,+}\}^{N^+_k}_{j=1}$, find~$\lambda^{k,+}$ by solving~$(\mathcal{P}(\mathcal{A}^+_k))$ and update
\[
\mu_{k+1}:=\sum^{N^+_k}_{j=1} \lambda^k_j \mu^{k,+}_j,~\mathcal{A}_{k+1}:= \mathcal{A}_k \setminus \left\{\,\mu^{k,+}_j\;|\;\lambda^{k,+}_j=0\,\right\}
\]
and increment~$k=k+1$.
\ENDFOR
\end{algorithmic}
\caption{Solution algorithm for~\eqref{def:problem}}
\label{alg:accgcg}
\end{algorithm}
 
\begin{rem}\label{rem:subproblems}
We point out that the functions~$|q_k(\cdot)|$ and~$\Psi_{q_k}$ are in general non-concave. Thus, in practice, computing their global maxima exactly can be infeasible. However, strategies based on multi-start gradient descent and heuristic rules have been successfully used and are widely accepted for minimization problems with total variation norm regularization, see e.g. \cite{bredies2013inverse, boyd2017alternating} and more general regularization functionals, see e.g. \cite{IglWal22,bredies2022generalized}. In this paper we use a basin-hopping-type algorithm \cite{wales1997global} whose performance is enough to compute the maximum of $|q_k(\cdot)|$ and~$\Psi_{q_k}$ efficiently and with satisfactory accuracy.
\end{rem}

\subsubsection{Sublinear convergence}\label{subsubsec:sublinear}
\begin{thm}\label{thm:sublinear}
Under the assumptions of the beginning of Section \ref{sec:minimizationproblem}, either Algorithm \ref{alg:accgcg} terminates after a finite number of steps outputting a minimizer of \eqref{def:problem}, or, denoting by $\mu_k$ the sequence generated by Algorithm \ref{alg:accgcg}, there exists a constant $C>0$ such that 
\[
J(\mu_k) - \inf_{\mu \in \mathcal{M}(\Omega)} J(\mu) \leq \frac{C}{1 + k}
\]
for all $k \in \N$. Moreover,  the generated  sequence $\mu_k$ admits at least a weak* accumulation point and every of such accumulation points are minimizers for \eqref{def:problem}. Finally, if \eqref{def:problem} has an unique minimizer $\bar \mu$, then the generated sequence $\mu_k$ converges weakly* to $\bar \mu$ as $k\rightarrow \infty$.
\end{thm}

\begin{proof}
The proof follows from a direct application of \cite[Theorem 3.4]{linconv} to Algorithm \ref{alg:accgcg}.
\end{proof}

\subsubsection{Linear convergence} \label{subsubsec:linear}
In this section, we prove that Algorithm~\ref{alg:accgcg} eventually converges linearly to a minimizer of Problem~\eqref{def:problem}. Reminiscent to Section~\ref{subsubsec:sublinear}, this statement will be derived by adapting the abstract convergence results from~\cite{linconv} to the current setting. Since this is rather technical, we omit the required proofs at this point and collect them in Appendix \ref{sec:appendix}. Through this section, we assume the following: 
\begin{ass}\label{ass:linearass}
There holds:
\begin{itemize}
    \item[$(\mathbf{B1})$] The functional~$F$ is strongly convex around the unique optimal observation~$\bar{y}$, i.e., there is a neighborhood~$\mathcal{N}(\bar{y})$ and~$\theta>0$ with
    \[
        (\nabla F(y_1)-\nabla F(y_2),y_1 -y_2)_{(Y^\ast, Y)} \geq \theta \|y_1 -y_2\|^2_Y \quad \text{for all}~y_1,y_2 \in \mathcal{N}(\bar{y}).
    \]
    \item [$(\mathbf{B2})$] There are~$\bar{N}_1,~\bar{N}_2 \in \mathbb{N} $,~$\bar{N}=\bar{N}_1+\bar{N}_2 >0$, as well as finite sets~$\left\{\bar{z}_i\right\}^{\bar{N}_1}_{i=1}\subset \operatorname{int} \Omega$ and~$\left\{(\bar{x}_j,\bar{y}_j)\right\}^{\bar{N}_2}_{j=1}\subset \operatorname{int} \Omega\times \operatorname{int}\Omega$ with
\[
  \left\{\,z \in \Omega\;|\;\bar{q}(z)=\pm \alpha\,\right\}=\left\{\bar{z}_i\right\}^{\bar{N}_1}_{i=1},\quad \left\{\,(x,y) \in \Omega \times \Omega\;|\;\Psi_{\bar{q}}(x,y)= 1\,\right\}=\left\{(\bar{x}_j,\bar{y}_j)\right\}^{\bar{N}_2}_{j=1}.
\]
and~$0<|\bar{x}_j-\bar{y}_j|^p< 2\alpha-\beta$.
    \item[$(\mathbf{B3})$] For every~$y\in Y$ there holds~$K_*y \in \operatorname{Lip}(\Omega)$ and the mapping~$K_* \colon Y \to \operatorname{Lip}(\Omega)$ is continuous. Moreover, we have~$\bar{q}\in \mathcal{C}^2(\Omega)$, and
    \[
        \operatorname{det}\left( \nabla^2 \bar{q}(\bar{z}_i) \right) \neq 0, \quad \operatorname{det}\left( \nabla^2 \Psi_{\bar{q}} (\bar{x}_j,\bar{y}_j) \right) \neq 0 \quad \text{for all}~i=1, \dots, \bar{N}_1,~j=1, \dots, \bar{N}_2.
    \]
    \item[$(\mathbf{B4})$] The set $\left\{K\left(\delta_{\bar{z}_i}\right)\right\}^{\bar{N}_1}_{i=1} \cup  \left\{K\left(\mathcal{D}_\beta(\bar{x}_j, \bar{y}_j)\right)\right\}^{\bar{N}_2}_{j=1}$ is linearly independent.
\end{itemize}
\end{ass}
\begin{rem}
By Theorem~\ref{thm:firstorder}, there holds~$|\bar{q}(z)|\leq \alpha$ and~$\Psi_{\bar{q}}(x,y)\leq 1$ for all~$z \in \Omega$ and~$(x,y) \in \Omega \times \Omega$. As a consequence, every~$(\bar{x},\bar{y})$ with~$\Psi_{\bar{q}}(\bar{x},\bar{y})=1$ necessarily satisfies 
\begin{align*}
\beta+|\bar{x}-\bar{y}|^p=\bar{q}(\bar{x})-\bar{q}(\bar{y}) \leq 2 \alpha,
\end{align*}
so Assumption~$(\mathbf{B2})$ requires only that in addition this inequality is strict.
\end{rem}
These assumptions imply the existence of a unique and sparse solution to~\eqref{def:problem}.
\begin{lemma} \label{lem:uniqueness}
Assumptions~$(\mathbf{B2})$ and $(\mathbf{B4})$ imply that the solution~$\bar{\mu}$ to~\eqref{def:problem} is unique and of the form
\[
    \bar{\mu}=\sum^{\bar{N}_1}_{i=1} (\bar{q}(\bar{z}_i)/\alpha)\bar\zeta_i \delta_{\bar{z}_i}+\sum^{\bar{N}_2}_{j=1} \bar\lambda_j \mathcal{D}_\beta(\bar{x}_j, \bar{y}_j)
\]
for some coefficients~$\bar{\zeta}_i, \bar{\lambda}_j \geq 0$,~$i=1,\dots,\bar{N}_1$,~$j=1,\dots,\bar{N}_2$.
\end{lemma}
\begin{proof}
For~$(\zeta,\lambda) \in \R^{\bar{N}_1}_+ \times \R^{\bar{N}_2}_+$ define the parametrized measure
\begin{align*}
    \mu(\zeta,\lambda)=\sum^{\bar{N}_1}_{i=1} (\bar{q}(\bar{z}_i)/\alpha)\zeta_i \delta_{\bar{z}_i}+\sum^{\bar{N}_2}_{j=1} \lambda_j \mathcal{D}_\beta(\bar{x}_j, \bar{y}_j).
\end{align*}
Then, due to the convexity and one-homogeneity of $\|\cdot\|_{\KR_p^\ab}$, we have
\begin{align}
    \|\mu(\zeta,\lambda)\|_{\KR_p^\ab} \leq \sum^{\bar{N}_1}_{i=1} \zeta_i + \sum^{\bar{N}_2}_{j=1} \lambda_j \quad \text{for all}~(\zeta,\lambda) \in \R^{\bar{N}_1}_+ \times \R^{\bar{N}_2}_+ .
\end{align}
Together with Proposition~\ref{e}, we conclude that every solution to~\eqref{def:problem} is of the form~$\bar\mu= \mu(\bar{\zeta},\bar{\lambda})$ where~$(\zeta,\lambda) \in \R^{\bar{N}_1}_+ \times \R^{\bar{N}_2}_+$ is a minimizer to 
\begin{align} \label{eq:subprobfinitehelp}
    \min_{(\zeta,\lambda) \in \R^{\bar{N}_1}_+ \times \R^{\bar{N}_2}_+} \left \lbrack F(K\mu(\zeta,\lambda))+ \sum^{\bar{N}_1}_{i=1} \zeta_i + \sum^{\bar{N}_2}_{j=1} \lambda_j \right \rbrack.
\end{align}
Finally, due to the strict convexity of~$F$ as well as the linear independence assumption~$(\mathbf{B4})$, the objective functional in~\eqref{eq:subprobfinitehelp} is also strictly convex. Thus, its minimizer and, as consequence, the solution to~\eqref{def:problem}, are unique.
\end{proof}
Finally, we assume~\textit{strong complementarity}, i.e., the unique coefficients in Lemma~\ref{lem:uniqueness} are positive:
\begin{itemize}
\item [$(\mathbf{B5})$] There holds $\bar{\zeta}_i, \bar{\lambda}_j > 0$,~$i=1,\dots,\bar{N}_1$,~$j=1,\dots,\bar{N}_2$.
\end{itemize}
As a consequence of Assumptions~$(\mathbf{B1})-(\mathbf{B5})$, we have the following linear convergence result.
\begin{thm}\label{thm:linear}
Let Assumptions~$(\mathbf{B1})-(\mathbf{B5})$ hold. Then Algorithm \ref{alg:accgcg} either terminates after a finite number of steps outputting a minimizer of \eqref{def:problem}, or, denoting by $\mu_k$ the sequence generated by Algorithm \ref{alg:accgcg}, there exists a constant $C>0$ such that 
\[
J(\mu_k) - \min_{\mu \in \mathcal{M}(\Omega)} J(\mu) \leq C \zeta^k
\]
for all $k \in \N$ large enough and some~$\zeta \in(0,1)$. Moreover, there holds~$\mu_k \xrightharpoonup{\ast} \bar{\mu}$.
\end{thm}

\section{Numerical examples}\label{sec:examples}
In this section we show that the KR-norm can be used successfully in the modelling of sparse optimal design tasks.
We consider the problem of reconstructing a signed measure $\tilde \mu \in \mathcal{M}(\Omega)$ from a collection of undetermined measurements $y$ that are modelled using a linear operator $K : \mathcal{M}(\Omega) \rightarrow Y$ mapping to a Hilbert space $Y$. Our goal is to design a model that is able to incorporate additional a priori information on the location of the signed measure $\tilde \mu$ using the KR-norm.
With this in mind we set up a variational problem that penalizes the fidelity to the measurements and the KR-norm of the difference between the unknown and a given positive reference measure $\mu_r$:
\begin{equation} \label{def:problemnum}
    \inf_{\tilde \mu \in \mathcal{M}(\Omega)}  \gamma \frac{1}{2}\|K\tilde \mu - y\|_Y^2 +\|\tilde \mu - \mu_r\|_{\KR_p^\ab} 
\end{equation}
for a suitable parameter choice $\gamma, \alpha, \beta >0$.
Note that \eqref{def:problemnum} can be rewritten in the form of Section \ref{sec:minimizationproblem} by the substitution $\mu = \tilde \mu - \mu_r$, obtaining the equivalent variational problem 
\begin{equation} \label{def:problemnum2}
    \inf_{\mu \in \mathcal{M}(\Omega)}  \gamma \frac{1}{2}\|K\mu - y + K\mu_r\|_Y^2 +\|\mu\|_{\KR_p^\ab}. \tag{$\mathscr{P}$}
\end{equation}
In this formulation and assuming that the observations are approximately consistent with some nonnegative ground truth measure (i.e.~if $y=K\mu^\dag + n$ with $\mu^\dag \geq 0$ and $n$ small), one can see that the negative part of the reconstructed measure $\mu$ is driven to be close to the reference measure $\mu_r$, while the nearly positive sum $\mu+\mu_r$ is encouraged to be faithful to the measurement $y$ through the action of $K$.
Moreover, depending on the choice of the weights $\alpha,\beta,\gamma$, penalizing the KR-norm of $\mu$ either favours the transport between the positive and negative parts of $\mu$ and thus preservation of mass, or TV-like regularization and thus creation of mass. Problem \eqref{def:problemnum2} can be interpreted as the reconstruction of a transport between the reference measure and a target measure (accessed only through the measurements) where creation of mass is allowed.



\begin{rem}
The inverse problem \eqref{def:problemnum} is suitable to model various practical problems. The reference measure $\mu_r$ can be thought as the initial distribution of some commodities that need to be transported in certain quantity to locations described by the measurement $y$. 
The linear operator $K$, for example, could describe either the linear diffusion of a source, the wave propagation from a source or more generally any  source-to-observation operator.
Our model allows for starting and  target measures having different mass and it indicates which portion of the commodity should be transported and which, if necessary at all, should be created in order to satisfy the demand. 
\end{rem}

In our first experiment, we consider measures defined in a $1$-dimensional domain $\Omega = [0,20] \subset \R$. We restrict our attention to a measurement operator $K : \mathcal{M}(\Omega) \rightarrow \mathbb{R}^{30}$ that is mapping a measure $\mu \in \mathcal{M}(\Omega)$ to the distributional solution of the heat equation in $\R$ with source $\mu$, i.e.
\begin{align}\label{eq:heat}
\left\{    \begin{array}{ll}
    \partial_t u - \Delta u = 0     & \text{in } (0,T) \times \R   \\
     u(0) = \mu     & 
    \end{array}
\right.
\end{align}
evaluated at time $T = 0.045$ and at locations $\{x_1, \ldots, x_{30}\}$ evenly spaced in $[0,20]$.
Note that in \eqref{eq:heat} the source $\mu$ has to be interpreted as the extension by zero of $\mu \in \mathcal{M}(\Omega)$ to $\R$. 
In this setting $K :\mathcal{M}(\Omega) \rightarrow \mathbb{R}^{30}$ can be computed through the convolution with the heat kernel as
\begin{align*}
    (K\mu)_i = \frac{1}{\sqrt{4\pi T}}\int_\Omega e^{\frac{-|x_i-y|^2}{4T}}\, d\mu(y) 
\end{align*}
for $i = 1, \ldots, 30$. Choosing $F(\mu) = \gamma \frac{1}{2}\|K\mu- y + K\mu_r\|_2^2$ for a given measurement $y \in \R^{30}$,
it is easy to verify that both $F$ and $K$ satisfy the assumptions of Section \ref{sec:minimizationproblem}.
We consider as reference measure $\mu_r = 2.8 \delta_7 + 2.8 \delta_{13}$ and as measurement $y \in \R^{30}$ the vector $K\mu^\dagger$, where 
$\mu^\dagger = \sum_{i=1}^{30} \delta_{x_i}$.
Finally, the parameters regulating the effect of the KR-norm penalization are set to be $\gamma = 60$, $\alpha = 0.9$, $\beta = 0.4$ and $p=1$. We run Algorithm \ref{alg:accgcg} until the stopping criteria
\begin{align}\label{eq:stopping}
\max\Bigg(    \max_{z \in \Omega} \frac{|q_k(z)|}{\alpha} ,  \max_{(x,y)\in\Omega\times \Omega} \Psi_{q_k}(x,y)\Bigg) \leq 1 + \varepsilon
\end{align}
is satisfied (see Lemma \ref{lem:terminationcriterion}), where we set $\varepsilon = 10^{-10}$, which was
attained after around $65$ iterations.
Figure \ref{fig:reconstruction} reports the reconstruction obtained with these parameter choices. The red stems are the Dirac deltas, the blue ones are the rescaled dipoles, the green ones represent the reference measure $\mu_r$ and the crosses are the magnitude of the reconstructed measurements $K\mu_{\overline k}$ at the locations $x_i$, which would equal $y - K\mu_r$ if the reconstruction were perfect. We remark that, under this parameter choice, the reconstructed measure is made of dipoles in the proximity of the reference measure and of Dirac deltas far from it. This is the effect of the KR-norm penalization that is encouraging transport for measurements close to the reference measure and creation of mass far from the reference measure.
\begin{figure}[!ht]
\includegraphics[scale=0.52,center]{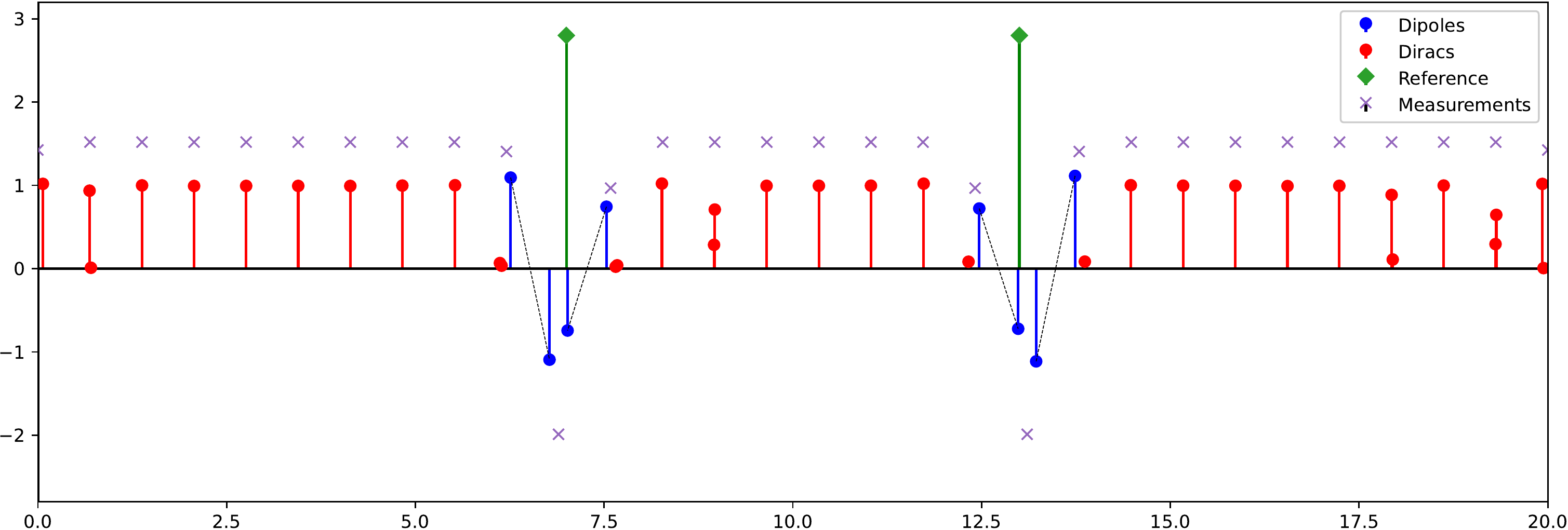}
\caption{\small Reconstruction of a minimizer of \eqref{def:problemnum2} by the application of Algorithm \ref{alg:accgcg}, with the crosses depicting $K \mu_{\bar k}$.}\label{fig:reconstruction}
\end{figure}

\begin{figure}[!ht]
\includegraphics[scale=0.45,center]{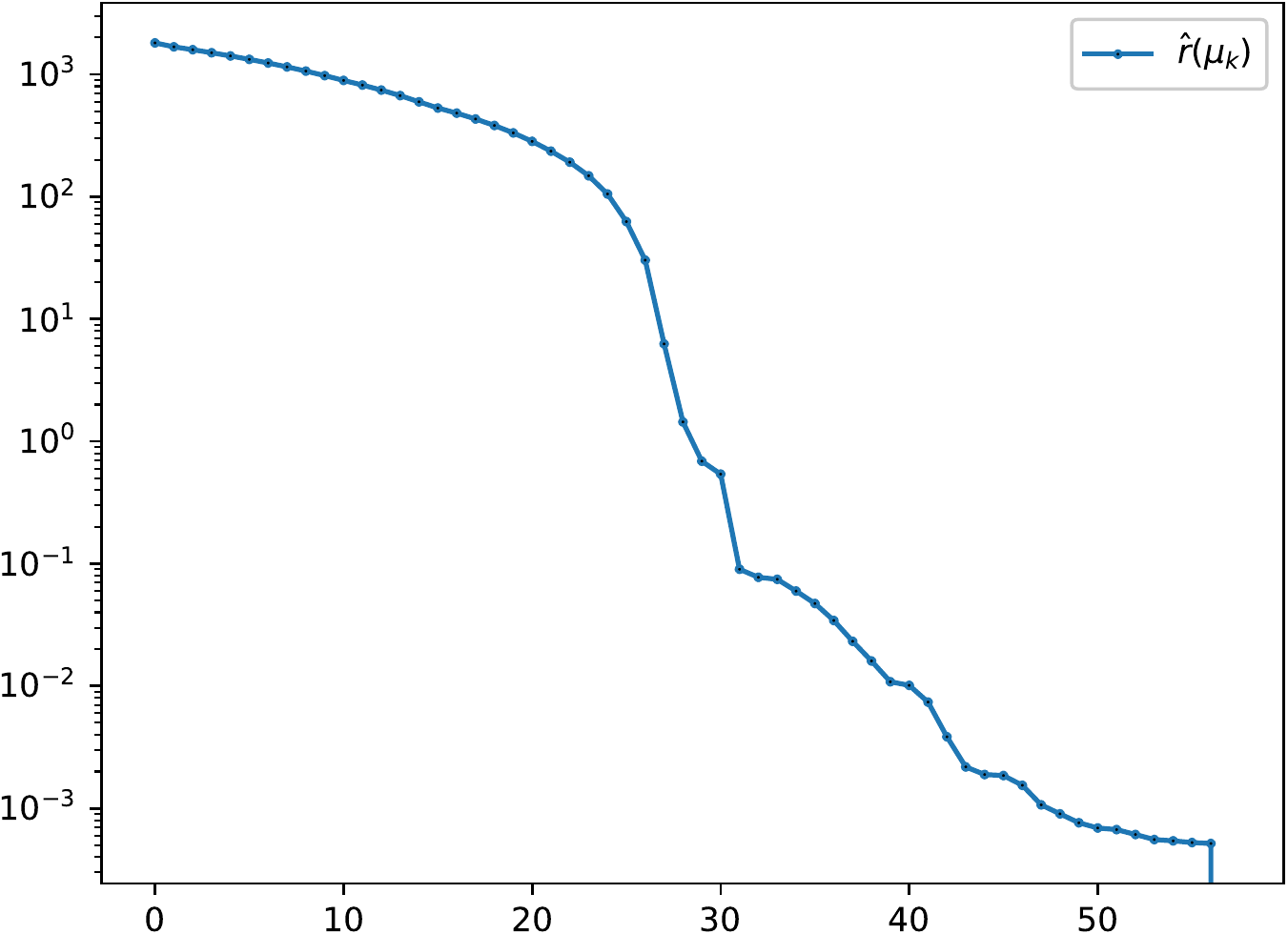}
\caption{\small Plot of the approximate residual $\hat r(\mu_k)$ in logarithmic scale.}\label{fig:convergencegraph}
\end{figure}
In Figure \ref{fig:convergencegraph} we graph the approximate residual $\hat r(\mu_k)$ defined as
\begin{align}\label{eq:approxre}
    \hat r(\mu_k) := F(K\mu_k) + \sum_{j=1}^{N_k} \lambda^k_j - \Big( F(K\mu_{\bar k}) + \sum_{j=1}^{N_{\bar k}} \lambda^{\bar k}_j\Big)
\end{align}
where $\mu_k = \sum_{j=1}^{N_k} \lambda_j^k \mu_j^k$ is the $k$-th iteration and  $
\mu_{\bar k} = \sum_{j=1}^{N_k} \lambda_j^{\bar k} \mu_j^{\bar k}$ is the output of the algorithm produced at the $\bar k$-th iteration. Note that $\hat r(\mu_k)$ is approximately an upper bound of the true residual $r(\mu_k) = J(\mu_k) - \inf_{\mu \in \mathcal{M}(\Omega)}J(\mu)$. Indeed, Theorem \ref{thm:sublinear} together with \cite[Theorem 4.4]{linconv} guarantees that the quantity $F(K\mu_{\bar k}) + \sum_{j=1}^{N_{\bar k}} \lambda^{\bar k}_j$ approximates $\inf_{\mu \in \mathcal{M}(\Omega)}J(\mu)$. Moreover the one-homogeneity and the subadditivity of the KR-norm implies that $J(\mu_k) \leq F(K\mu_k) + \sum_{j=1}^{N_k} \lambda^k_j$.

\begin{figure}[!ht]
\begin{minipage}{0.40\textwidth}
\vspace*{3mm}
\begin{flushleft}
\includegraphics[width=\linewidth]{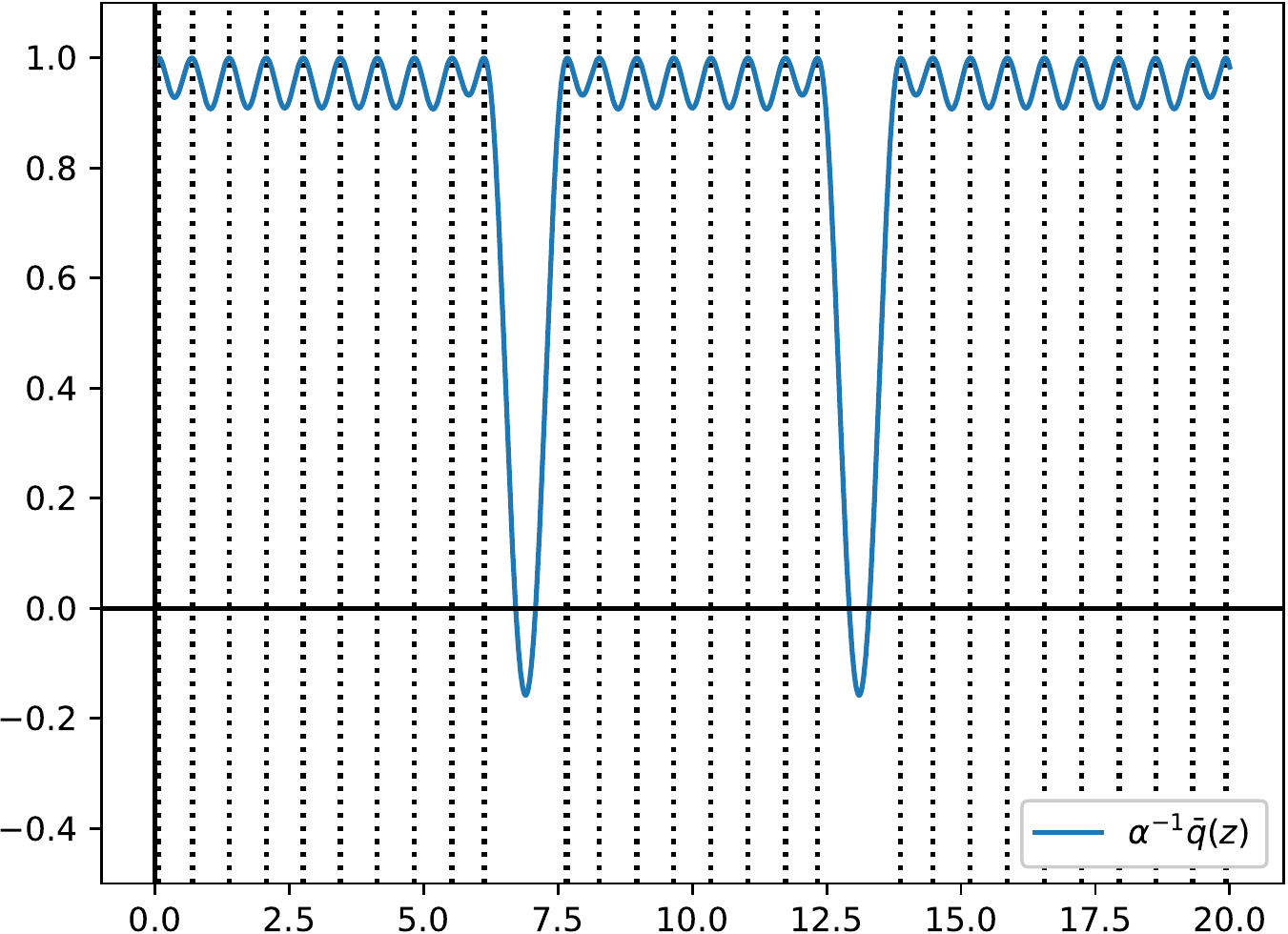}
\end{flushleft}
\end{minipage} 
\hspace*{0.1mm}
\begin{minipage}{0.58\textwidth}
\vspace*{3mm}
\begin{flushleft}
\includegraphics[width=\linewidth]{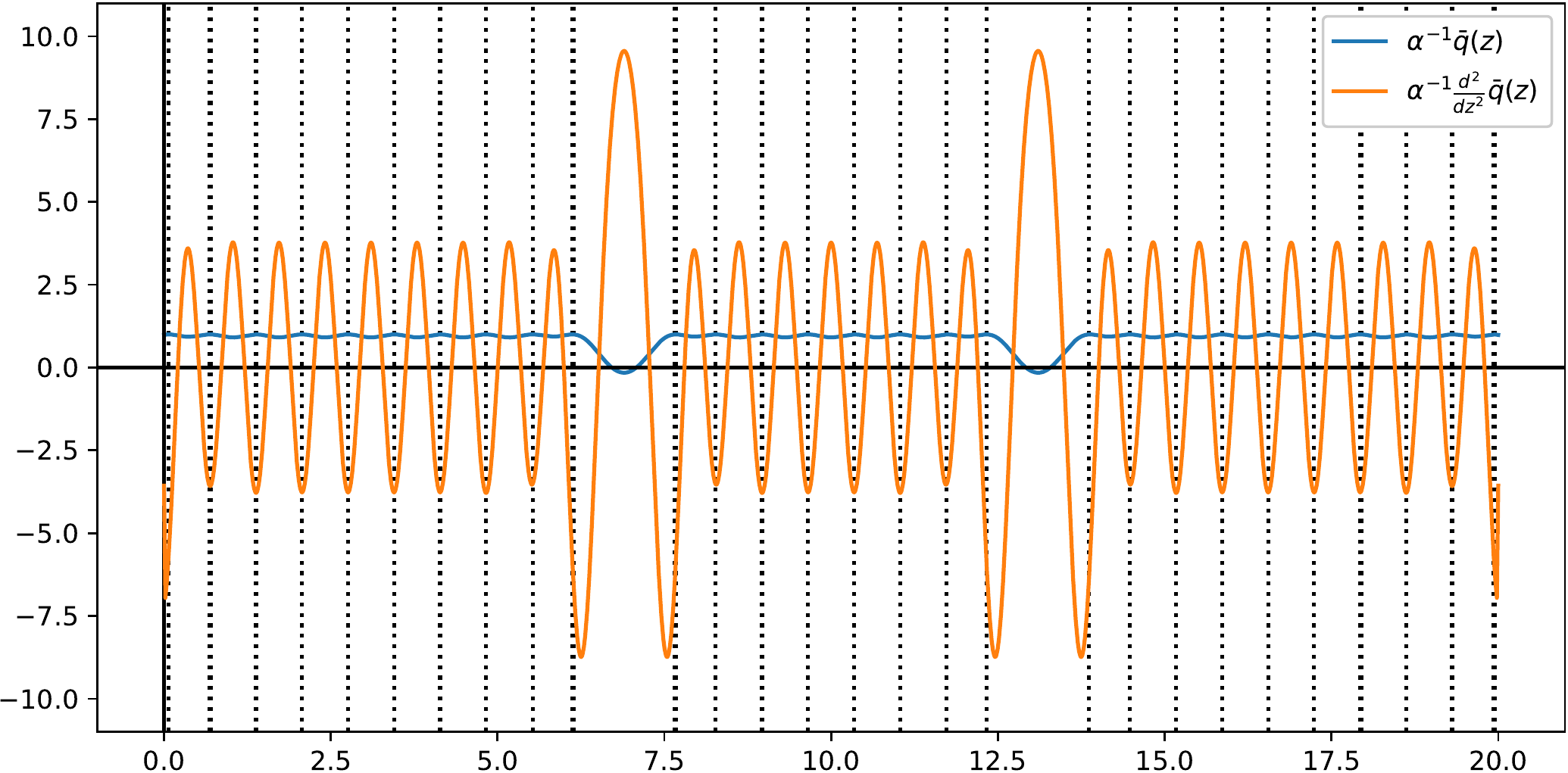}
\end{flushleft}
\end{minipage} 
\caption{\small On the left the rescaled dual variable $\bar q(z)/\alpha$ computed in the output of Algorithm \ref{alg:accgcg}. On the right $\bar q(z)/\alpha$ and its second derivative. The vertical dotted lines, both on the left and on the right figure, represent the location of the Dirac deltas appearing in the output of Algorithm \ref{alg:accgcg}.} \label{fig:dualsecond}
\end{figure}

\begin{figure}[!ht]
\begin{minipage}{0.50\textwidth}
\vspace*{3mm}
\begin{flushleft}
\includegraphics[width=0.95\linewidth]{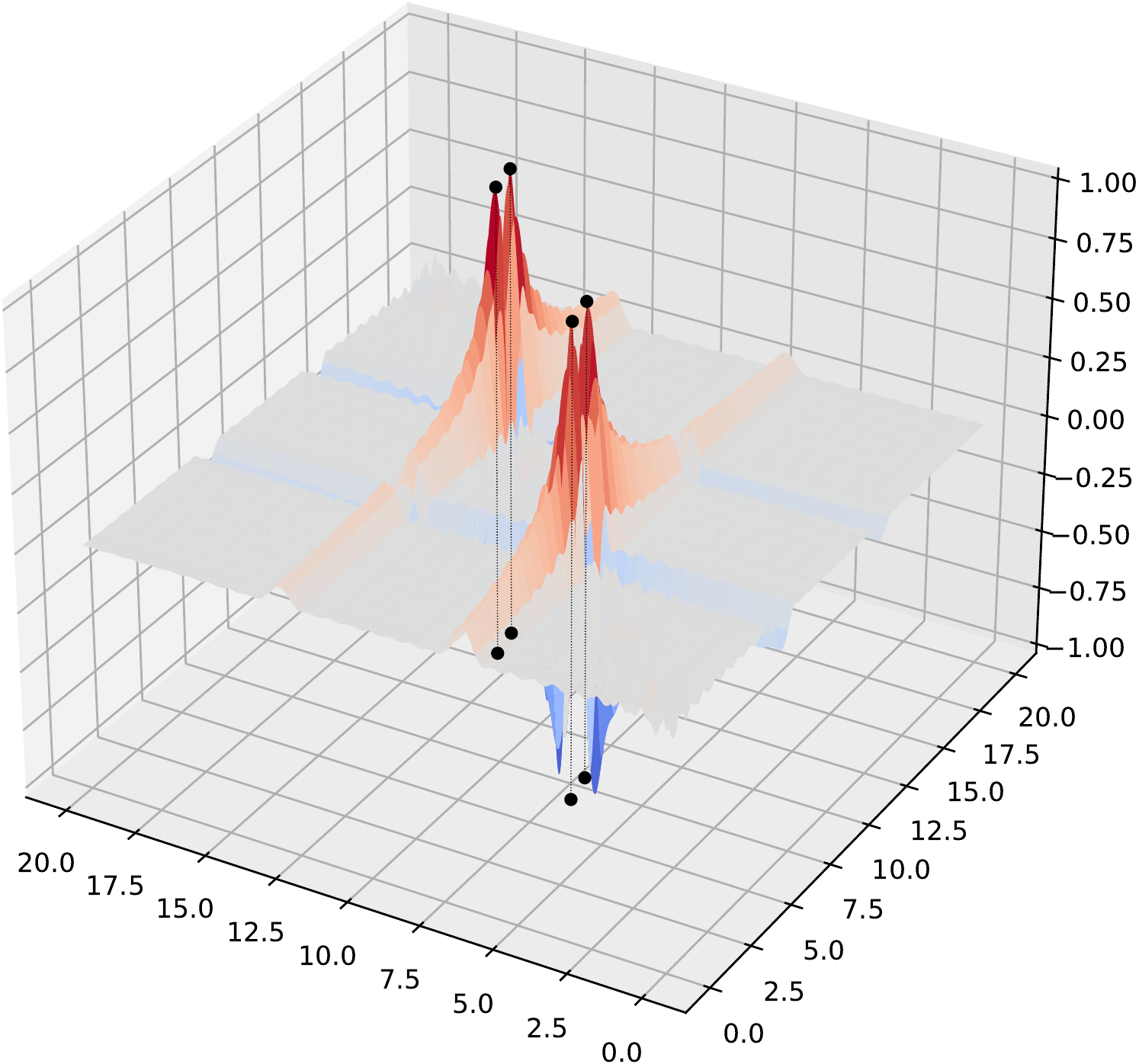}
\end{flushleft}
\end{minipage} 
\hspace*{8mm}
\begin{minipage}{0.42\textwidth}
\vspace*{3mm}
\footnotesize{
\begin{tabular}{ |p{3.8cm}||p{2cm}|  }
 \hline
 Dipole location &$\operatorname{det}\left( \nabla^2 \Psi_{\bar{q}}\right)$ \\
 \hline
 $(6.26,6.78)$ & 73.52 \\
  $(7.53,7.02)$ &   73.46 \\
  $(12.47, 12.98)$    & 73.42 \\
 $(13.74,13.22)$    & 73.43 \\
  \hline
\end{tabular}\\

\medskip
\centering
\includegraphics[width=0.9\linewidth]{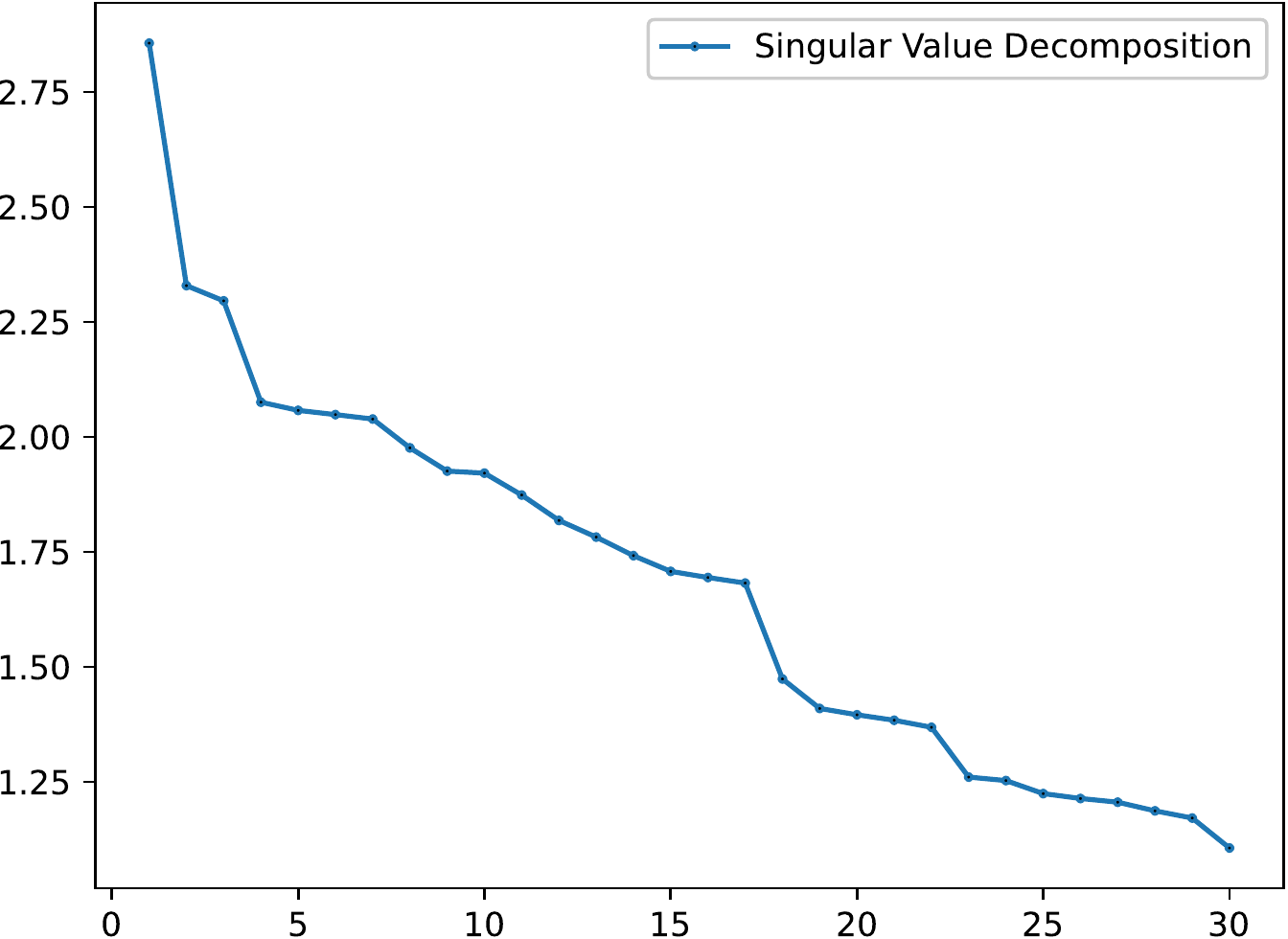}

}
\end{minipage} 
\caption{\small On the left the plot of $\Psi_{\bar q}$. The dotted lines are drawn for each dipole in the output measure of the algorithm. On the right, to demonstrate assumptions $(\mathbf{B3})$ and $(\mathbf{B4})$, a table with the location of each dipole and the corresponding $\operatorname{det}\left( \nabla^2 \Psi_{\bar{q}}\right)$, and the singular values of the matrix constructed from $\left\{K\left(\delta_{\bar{z}_i}\right)\right\}^{\bar{N}_1}_{i=1}$ and $\left\{K\left(\mathcal{D}_\beta(\bar{x}_j, \bar{y}_j)\right)\right\}^{\bar{N}_2}_{j=1}$.}\label{fig:Psi}
\end{figure}

As shown on the convergence graph in Figure \ref{fig:convergencegraph}, the rate of convergence is at least linear in practice. This is theoretically ensured by Theorem \ref{thm:linear} provided Assumptions $(\mathbf{B1})$--$(\mathbf{B5})$ in Section \ref{subsubsec:linear} are fulfilled. Clearly $(\mathbf{B1})$ is verified by our choice of $F$. Since, we do not know the explicit minimizer, in order to verify the validity of $(\mathbf{B2})$ and $(\mathbf{B3})$ we consider the dual variable $\bar q$ of the output of Algorithm \ref{alg:accgcg}, we plot the functions $\bar q$, $\Psi_{\bar q}$ and we compute their second derivatives. Figure \ref{fig:dualsecond} and Figure \ref{fig:Psi} shows that $\bar q = \alpha$ precisely on the location of the Dirac deltas of the output measure $\mu_{\bar k}$ and $\Psi_{\bar q} = 1$ on the location of its dipoles. Moreover, Figure \ref{fig:dualsecond} demonstrates that $\bar q''(z) < 0$ on the Dirac deltas and Figure \ref{fig:Psi} reports the values of $\operatorname{det}\left( \nabla^2 \Psi_{\bar{q}}\right)$ on the dipoles, showing that they are strictly positive, and the singular values of the matrix constructed from $\left\{K\left(\delta_{\bar{z}_i}\right)\right\}^{\bar{N}_1}_{i=1}$ and $\left\{K\left(\mathcal{D}_\beta(\bar{x}_j, \bar{y}_j)\right)\right\}^{\bar{N}_2}_{j=1}$ to check their linear independence and ensuring Assumption $(\mathbf{B4})$. Finally, it is easy to verify that for $y \in \R^{30}$
\begin{align}
z \mapsto K_* y(z) =  \frac{1}{\sqrt{4\pi T}} \sum_{i=1}^{30} e^{-\frac{|x_i-z|^2}{4T}}y_i \in \text{Lip}(\Omega)
\end{align}
and the mapping $K_* : \R^{30} \rightarrow \text{Lip}(\Omega)$ is continuous as required in $(\mathbf{B3})$

As second experiment we consider infinite dimensional measurements of the same convolution with the heat kernel, so that $K : \mathcal{M}(\Omega) \rightarrow L^2(\Omega)$ for $\Omega=[0,20]$ is now defined as 
\begin{align*}
    (K\mu)(x) = \frac{1}{\sqrt{4\pi T}}\int_\Omega e^{\frac{-|x-y|^2}{4T}}\, d\mu(y) \quad x \in \Omega.
\end{align*}
We choose $F(\mu) = \gamma \frac{1}{2}\|K\mu- y + K\mu_r\|_{L^2(\Omega)}^2$ with reference measure $\mu_r = 1.5 \delta_8 + 1.5 \delta_{12}$ and measurement $y(x) = \sin(\frac{\pi x}{4}) + 1$.
The parameters are set to be $\gamma = 4$, $\alpha = 0.8$, $\beta = 0.3$ and $p=1$. We run Algorithm \ref{alg:accgcg} until the stopping criteria \eqref{eq:stopping}
is satisfied with $\varepsilon = 10^{-6}$, which was attained after around $190$ iterations. 
Figure \ref{fig:reconstruction2} reports the reconstruction obtained with these parameter choices while Figure \ref{fig:dualgrah2} shows the dual variable of the output of the algorithm and the convergence graph of the approximate residual defined in \eqref{eq:approxre}. 

\begin{figure}[!ht]
\includegraphics[scale=0.52,center]{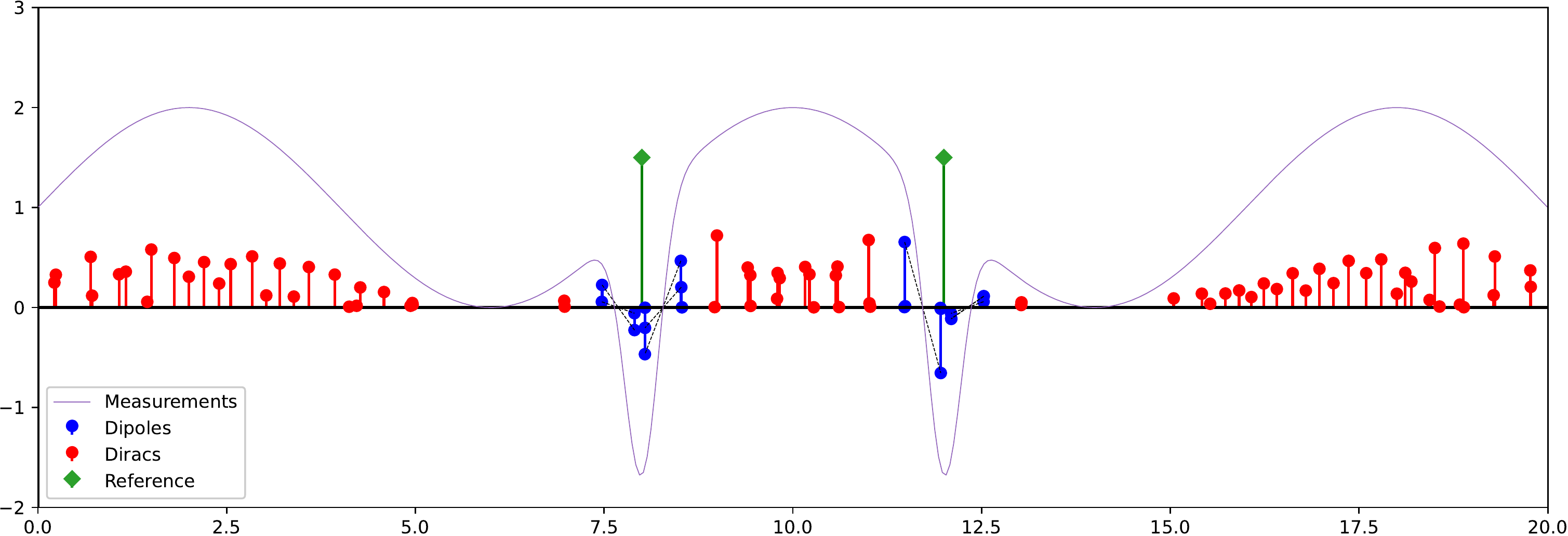}
\caption{\small Reconstruction of a minimizer of \eqref{def:problemnum2} for infinite dimensional measurements by the application of Algorithm \ref{alg:accgcg}, with a continuous plot showing $K \mu_{\bar k}$.}\label{fig:reconstruction2}
\end{figure}

\begin{figure}[!ht]
\begin{minipage}{0.49\textwidth}
\vspace*{3mm}
\begin{flushleft}
\includegraphics[scale=0.45,center]{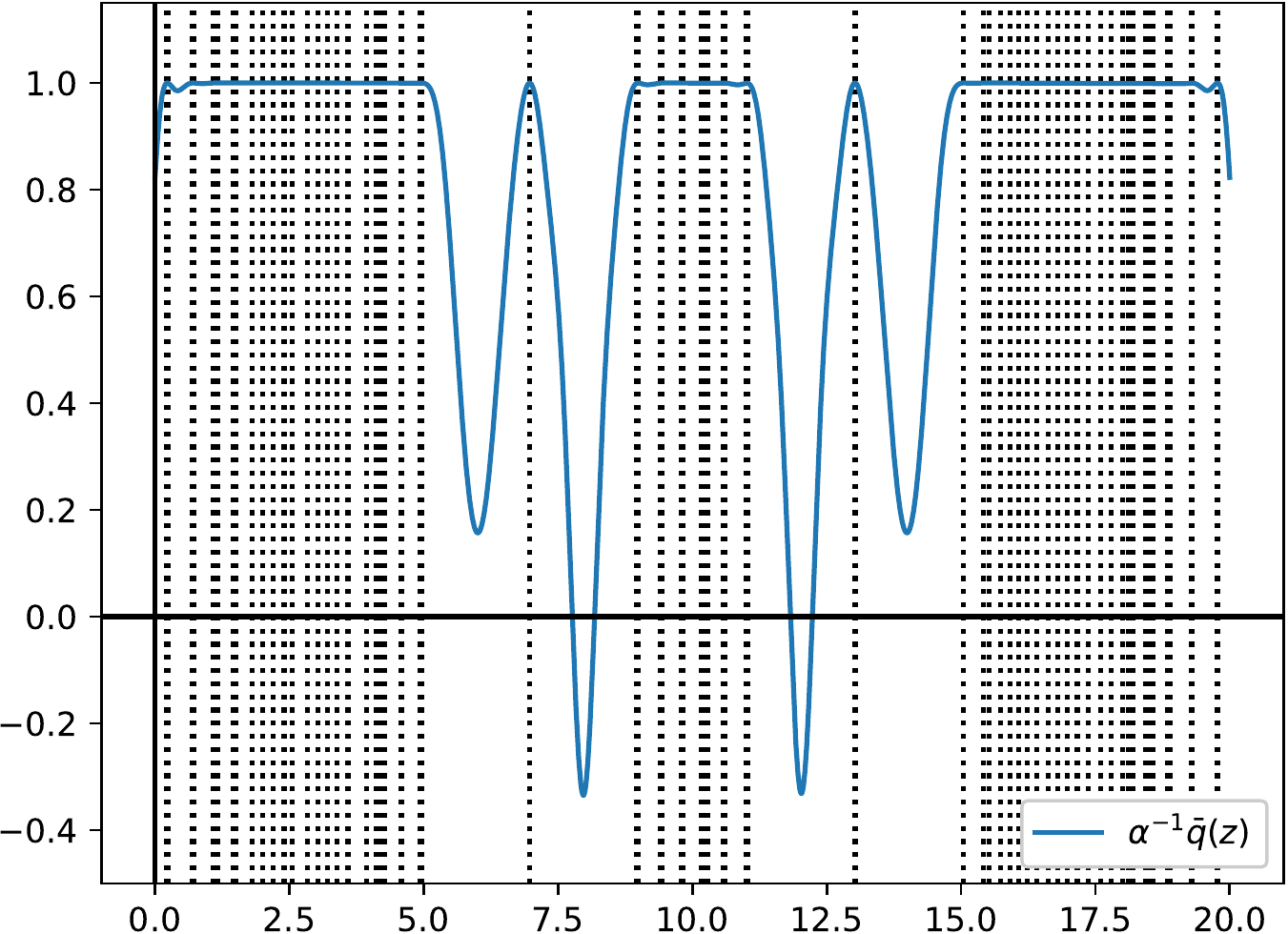}
\end{flushleft}
\end{minipage} 
\hspace*{0.1mm}
\begin{minipage}{0.49\textwidth}
\vspace*{3mm}
\begin{flushleft}
\includegraphics[scale=0.45,center]{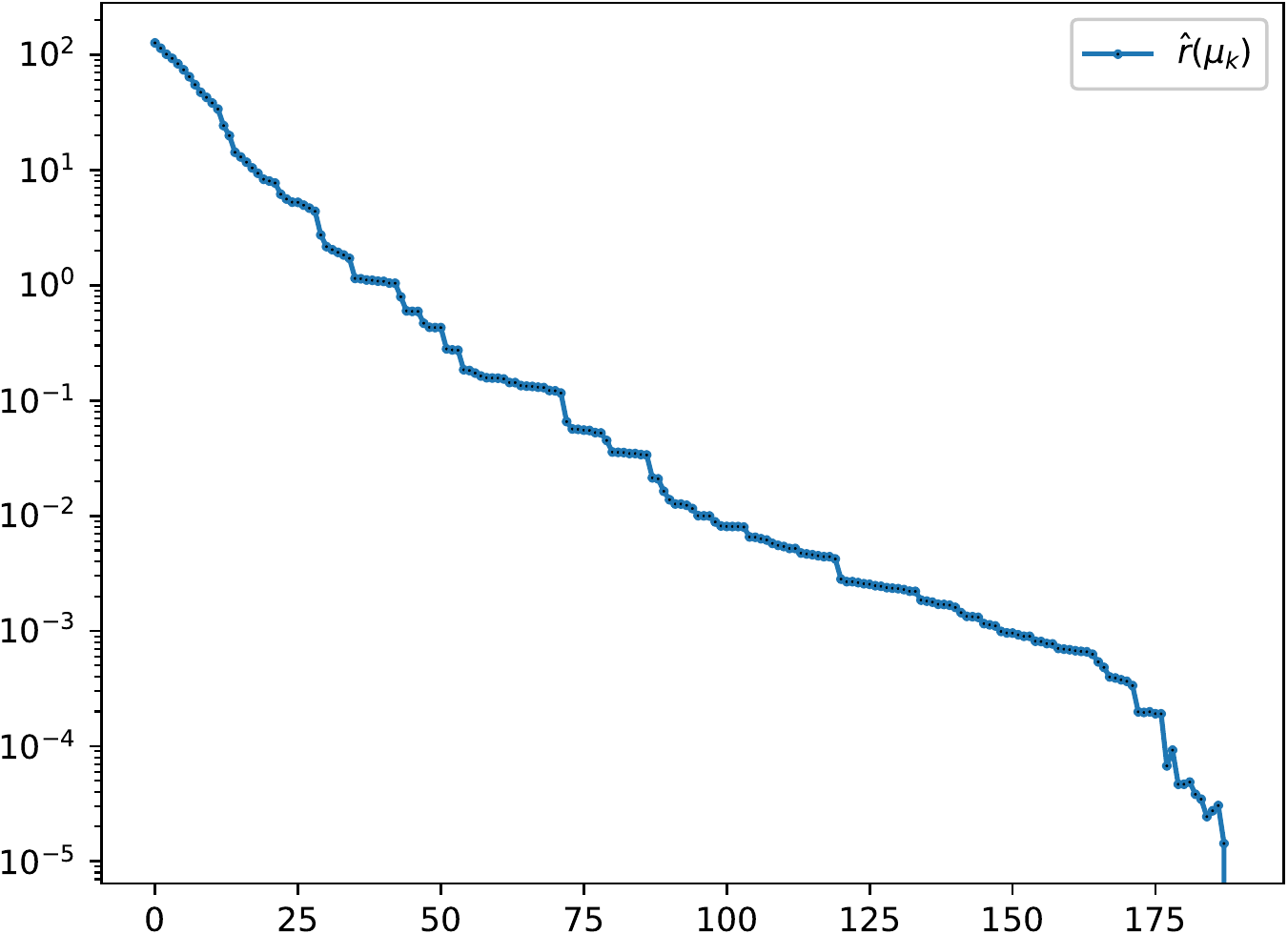}
\end{flushleft}
\end{minipage} 
\caption{\small On the left the rescaled dual variable $\bar q(z)/\alpha$ computed for the output of Algorithm \ref{alg:accgcg}. On the right the approximate residual $\hat r(\mu_k)$ in logarithmic scale.} \label{fig:dualgrah2}
\end{figure}

All of the previous experiments are carried out on Python3 on
a MacBook Pro with 8 GB RAM and an Intel®Core™ i5, Quad-Core, 2.3 GHz.

\appendix
\section{Proofs for Section~\ref{subsubsec:sublinear}}\label{sec:appendix}
In this section, we collect the necessary auxiliary results for the proof of Theorem~\ref{thm:linear} by applying the results of~\cite{linconv}. For this purpose, we keep using the notation $B:=\{\mu \,|\, \|\mu\|_{\KR_p^{\alpha,\beta}}\leq 1\}$ and further introduce~$\mathcal{B}\coloneqq \overline{\operatorname{Ext}(B)}^*$. Since the predual space~$\mathcal{C}(\Omega)$ is separable,~$\mathcal{B}$ is weak* compact and there exists a metric~$d_\mathcal{B}$ which metrizes the weak* topology on~$\mathcal{B}$, see~\cite[Theorem~3.29]{brezis}.

\begin{lemma} \label{lem:charclosedness}
We have
\begin{align*}
    \mathcal{B}= \{\,(\sigma/\alpha)\delta_z\;|&\;\sigma\in\{-1,+1\},~z \in \Omega\,\} \\ &\cup \{\,\mathcal{D}_\beta(x,y)\;|\;(x,y) \in \Omega \times \Omega,~0\leq|x-y|^p\leq 2 \alpha-\beta\,\}.
\end{align*}
\end{lemma}
\begin{proof}
By the characterization of~$\operatorname{Ext}(B)$ we first observe that
\begin{align*}
\mathcal{B}= \overline{\{\,(\sigma/\alpha)\delta_z\;|\;}&\overline{\sigma\in\{-1,+1\},~z \in \Omega\,\}}^*\\ &  \cup \overline{\left\{\,\mathcal{D}_\beta(x,y)\;|\;(x,y) \in \Omega \times \Omega,~0<|x-y|^p< 2 \alpha-\beta\,\right\}}^*.    
\end{align*}
Now, let~$\mu_k=(\sigma_k/\alpha) \delta_{z_k}$,~$\sigma_k\in\{-1,1\}$, $z_k\in \Omega$,~$k \in \N$, denote a weak* convergent sequence with limit~$\bar{\mu}$. Then, due to the compactness of~$\Omega$, there exists a subsequence, denoted by the same symbol, with
\[
    (\sigma_k, z_k) \rightarrow (\bar{\sigma}, \bar{z}) \quad \text{for some}~(\bar{\sigma}, \bar{z}) \in \{-1,1\} \times \Omega.
\]
Setting~$\tilde{\mu}=(\bar{\sigma}/\alpha)\delta_{\bar{z}}$, the associated sequence of measures satisfies
\[
    \langle q,\mu_k \rangle= (\sigma/\alpha) q(z_k) \rightarrow (\bar\sigma/\alpha) q(\bar z)=\langle q,\tilde \mu \rangle \quad \text{for all}~q \in \mathcal{C}(\Omega).
\]
Since weak* limits are unique,~$\bar{\mu}=\tilde \mu$ follows.

Similarly, we see that any weak* convergent sequence~$\mu_k=\mathcal{D}_\beta(x_k,y_k)$ with
\[
(x_k,y_k)\in\Omega\times\Omega,~0<|x_k-y_k|^p< 2\alpha-\beta   
\]
necessarily satisfies~$\mu_k \xrightharpoonup{\ast} \mathcal{D}_\beta(\bar{x},\bar{y})$ for some~$(\bar{x},\bar{y})\in \Omega \times \Omega$ with~$0\leq |\bar{x}-\bar{y}|^p\leq 2\alpha-\beta$. This finishes the proof.
\end{proof}
In order to apply the abstract convergence result of~\cite{linconv}, we have to check some structural assumptions. 
First, we show that, due to Assumption $(\mathbf{B2})$, the linear problem
\[
    \max_{\mu \in \mathcal{B}} \langle \bar{q}, \mu \rangle
\]
admits finitely many maximizers and all of them are extremal points. 
\begin{lemma} \label{lem:onlyfinite}
Let Assumption~$(\mathbf{B2})$ hold. Then we have
\[
    \argmax_{\mu \in \mathcal{B}} \langle \bar{q}, \mu \rangle= \left\{(\operatorname{sign}(\bar{q}(\bar{z}_i))/\alpha)\delta_{\bar{z}_i}\right\}^{\bar{N}_1}_{i=1} \cup \left\{\mathcal{D}_\beta(\bar{x}_j,\bar{y}_j)\right\}^{\bar{N}_2}_{j=1}.
\]
\end{lemma}
\begin{proof}
Define
\[
    D:=\big\{(\operatorname{sign}(\bar{q}(\bar{z}_i))/\alpha)\delta_{\bar{z}_i}\big\}^{\bar{N}_1}_{i=1} \cup \left\{\mathcal{D}_\beta(\bar{x}_j,\bar{y}_j)\right\}^{\bar{N}_2}_{j=1}.
\]
By assumption,~$D$ is nonempty and there holds~$\langle \bar{q}, \mu \rangle=1$ for all~$\mu \in D$.
Moreover, since~$\bar{q}$ is the unique dual variable for Problem~\eqref{def:problem} and~$ \|\cdot\|_{\KR_p^\ab}$ is positively one-homogeneous, we conclude
\[
    \max_{\mu \in \mathcal{B}}\langle \bar{q},\mu \rangle=1 \quad \text{and thus}~D \subset \argmax_{\mu \in \mathcal{B}} \langle \bar{q}, \mu \rangle
\]
The inverse inclusion follows immediately from Assumption~$(\mathbf{B2})$ which gives
\[
    \max_{z}|\bar{q}(z)|\leq \alpha,~\max_{(x,y)}\Psi_{\bar{q}}(x,y)\leq  1,
\]
as well as noting that
\begin{align*}
    \langle \bar{q},\ \sigma \delta_z  \rangle=1\text{ for }\sigma \in \{-1,1\}\text{ and }z \in \Omega &\text{ implies}~|p(z)|=\alpha ,\text{ and}\\
    \langle \bar{q},\ \mathcal{D}_\beta(x,y)  \rangle=1 \text{ with } 0\leq |x-y|\leq 2 \alpha-\beta   &\text{ is equivalent to }\Psi(x,y)=1.\qedhere
\end{align*}
\end{proof}
For abbreviation, set
\[
    \bar{\mu}^1_i= (\sign(\bar{q}(\bar{z}_i))/\alpha) \delta_{\bar{z}_i},~\bar{\mu}^2_j= \mathcal{D}_\beta(\bar{x}_j,\bar{y}_j) \quad \text{for all}~i=1,\dots, \bar{N}_1,~j=1,\dots \bar{N}_2.
\]
Second, we have to show the existence of~$d_{\mathcal{B}}$-neighborhoods~$U^1_i$ of~$\bar{\mu}^1_i$ and $U^2_j$ of~$\bar{\mu}^2_j$ in~$\mathcal{B}$, respectively, as well as of a mapping~$g \colon \operatorname{Ext}(B) \times \operatorname{Ext}(B)$ and~$\theta,~C_K >0$ with
\begin{equation} \label{def:LipandQuad}
    \|K(\mu-\mu^k_j)\|_Y \leq  C_K \, g(\mu, \mu^k_j)\ \text{ and }\ 1-\langle \bar{q},\mu \rangle \geq \theta\, g(\mu,\mu^k_j)^2
\end{equation}
for all~$j=1, \dots, \bar{N}_k$,~$k=1,2$, and all~$\mu \in U^k_j \cap \operatorname{Ext}(B)$.
We claim that this satisfied for
\begin{align*}
    g(\mu_1,\mu_2) \coloneqq \begin{cases} |z_1-z_2|+|\sigma_1-\sigma_2| & \mu_1= \sigma_1 \delta_{z_1},~\mu_1= \sigma_2 \delta_{z_2},~z_1,z_2 \in \Omega,~\sigma_1,\sigma_2 \in \{-1,1\} \\ \left|\begin{pmatrix} x_1-x_2 \\ y_1-y_2\end{pmatrix} \right| & \mu_1=\mathcal{D}_\beta(x_1,y_1),~\mu_2=\mathcal{D}_\beta(x_2,y_2),~ (x_1,y_1),(x_2,y_2) \in \Omega \times \Omega   \\
    0 & \text{else}
    \  \end{cases}.
\end{align*}
The proof is split into two parts. First, we characterize open~$d_{\mathcal{B}}$-neighborhoods around the associated extremal points.
\begin{lemma} \label{lem:openneighbor}
For~$0< R$ define the sets
\[
    U^1_i(R) \coloneqq \left\{\,(\operatorname{sign}(\bar{q}(\bar{z}_i))/\alpha)\delta_{z}\;|\;z \in B_R(\bar{z}_i)\,\right\}  \quad \text{for all}~ i=1, \dots, \bar{N}_1,
\]
as well as
\[
    U^2_j(R) \coloneqq \left\{\,\mathcal{D}_\beta(\bar{x}_j,\bar{y}_j)\;|\;(x,y) \in B_R(\bar{x}_j)\times B_R(\bar{y}_j)\,\right\} \quad \text{for all}~ j=1, \dots, \bar{N}_2.
\]
Then ${U}^1_i(R)$ is a $d_{\mathcal{B}}$-neighborhood of~$(\operatorname{sign}(\bar{q}(\bar{z}_i)/\alpha)\delta_{\bar{z}_i}$,~$i=1,\dots,\bar{N}_1$, and $\bar{U}^2_j(R)$ is a $d_{\mathcal{B}}$-neighborhood of~$\mathcal{D}_\beta(\bar{x}_j,\bar{y}_j)$,~$j=1,\dots,\bar{N}_2$. Moreover, for every~$R>0$ small enough, there holds~${U}^1_i(R),~{U}^2_i(R) \subset \operatorname{Ext}(B)$.
\end{lemma}
\begin{proof}
Let indices~$i\in\{1, \dots, \bar{N}_1\}$ and~$j\in\{1, \dots, \bar{N}_2\}$ be arbitrary but fixed. We first show the claimed statement for~$\bar{U}^2_j$. Noting that~$(\mathcal{B},d_{\mathcal{B}})$ is a metric space, it suffices to show that any sequence~$\{\mu_k\}_k \subset \mathcal{B}$ with~$\mu_k \xrightharpoonup{\ast} \mathcal{D}_\beta(\bar{x}_j,\bar{y}_j)$ eventually lies in~$\bar{U}^2_j$ for all~$k \in \N$ large enough. For this purpose, assume that~$\{\mu_k\}_{k}$ admits a subsequence, denoted by the same symbol, of the form~$\mu_k=(\sigma_k/\alpha) \delta_{z_k}$ for some~$\sigma_k \in\{-1,1\},~z_k \in\Omega$. Then, by possibly selecting another subsequence, we get~$\mu_k \xrightharpoonup{\ast} (\bar{\sigma}/\alpha)\delta_{\bar{z}}$ for some~$\bar{\sigma} \in \{-1,1\},~\bar{z} \in \Omega$. Noting that weak* limits are unique and~$\bar{\sigma}\delta_{\bar{z}} \neq \mathcal{D}_\beta(\bar{x}_j, \bar{y}_j)$ yields a contradiction.
In the same way, we exclude the existence of a subsequence with~$\mu_k=0$ for all~$k$.
Hence, for all~$k\in\N$ large enough, we have~$\mu_k=\mathcal{D}_\beta(x_k,y_k)$ for some~$(x_k,y_k ) \in \Omega \times \Omega$ with~$0< |x_k,y_k|\leq 2\alpha-\beta$. By a similar contradiction argument,~$(x_k,y_k) \rightarrow (\bar{x}_j,\bar{y}_j)$ has to hold. Thus, for every~$k \in \N$ large enough, we have~$(x_k,y_k) \in B_{R_2}(\bar{x}_j,\bar{y}_j)$ and thus~$\mu_k \in \bar{U}^2_j$, finishing the proof.
The openness of~$\bar{U}^1_j$ follows by similar argument. In fact, if~$\{\mu_k\}_k  \subset \mathcal{B} $ satisfies
\[
 \mu_k \xrightharpoonup{\ast}   (\operatorname{sign}(\bar{q}(\bar{z}_i)/\alpha)\delta_{\bar{z}_i},
\]
then~$\mu_k=(\sigma_k/\alpha) \delta_{z_k}$,~$\sigma_k \in\{-1,1\},~z_k\in\Omega$ for all~$k$ large enough since~$\bar{\mu}^1_i \neq \mathcal{D}_\beta(x,y)$ for every~$(x,y)\in \Omega \times \Omega$. Moreover, from~\cite[Lemma~3.16]{linconv}, we get~$\sigma_k=$ for all~$k \in \N$ large enough. Finally, if there is a subsequence of~$\{z_k\}_k$, denoted by the same symbol, with~$z_k \rightarrow \bar{z}$ with~$\bar{z}\neq \bar{z}_i$, then we can choose~$\varphi \in \mathcal{C}(\Omega)$ satisfying~$\varphi(\bar{z})=0$ and~$\varphi(\bar{z}_i)=1$. For the corresponding subsequence of measures~$\mu_k$, we then obtain
\[
    \langle \varphi, \mu_k \rangle= (\sigma_k/\alpha) \varphi(z_k) \rightarrow (\sign(\bar{q}(\bar{z}_i))/\alpha) \varphi(\bar{z})=0 \neq \langle \varphi,\bar{\mu}_i \rangle
\]
yielding a contradiction and thus~$\bar{z}=\bar{z}_i$.
\end{proof}
Next we prove the Lipschitz and quadratic growth properties from~\eqref{def:LipandQuad}.
\begin{lemma} \label{lem:locallipschitz}
There are~$R_1,C_K>0$ with
\[
    \|K(\mu-\bar{\mu}^\ell_j) \|_Y \leq C_K \, g(\mu, \bar{\mu}^\ell_j)
\]
for all~$\mu \in U^\ell_j(R_1)$,~$j=1,\dots,\bar{N}_\ell$,~$\ell=1,2$.
\end{lemma}
\begin{proof}
By assumption,~$K_* \colon Y \to \operatorname{Lip}(\Omega)$ is continuous. As a consequence, we immediately get
\begin{align*}
    \|K(\delta_z-\delta_{\bar{z}_i}) \|_Y &= \sup_{\|v\|_Y \leq 1} \langle K_* v, \delta_z-\delta_{\bar{z}_i}\rangle= \sup_{\|v\|_Y \leq 1} \left \lbrack \lbrack K_* v \rbrack(z)-\lbrack K_* v \rbrack(\bar{z}_i)\right \rbrack
    \\ & \leq \|K_*\|_{Y, \operatorname{Lip}} |z-\bar{z}_i|
\end{align*}
for all~$z\in \Omega$.
For~$\mathcal{D}_\beta(\bar{x}_j, \bar{y}_j)$ we can argue similarly. For this purpose, if~$R_1>0$ is small enough, we have
\[
    |\bar{x}_i-\bar{y}_i|^p-|x-y|^p \leq c \left|\begin{pmatrix} x-\bar{x}_j \\ y-\bar{y}_j\end{pmatrix} \right|
\]
for all~$(\bar{x}_i,\bar{y}_i) \in B_R(\bar{x}_j) \times B_R(\bar{y}_j)$ since~$|\bar{x}_i-\bar{y}_i|>0$.
As a consequence, we get
\begin{align*}
    \|K(\mathcal{D}_\beta(x,y)-\mathcal{D}_\beta(\bar{x}_j,\bar{y}_j)) \|_Y&= \sup_{v \in Y} \langle K_* v, \mathcal{D}_\beta(x,y)-\mathcal{D}_\beta(\bar{x}_j,\bar{y}_j)\rangle \\ & = \sup_{y \in Y} \left \lbrack \frac{\lbrack K_*v \rbrack(x)-\lbrack K_*v \rbrack(y)}{\beta+ |x-y|^p}-\frac{\lbrack K_*v \rbrack(\bar{x}_j)-\lbrack K_*v \rbrack(\bar{y}_j)}{\beta+ |\bar{x}_j-\bar{y}_j|^p} \right \rbrack
    \\& \leq D_1+ D_2 
\end{align*}
where we abbreviate
\begin{align*}
    D_1 &\coloneqq \frac{\|K_*\|_{Y,\operatorname{Lip}}(|x-\bar{x}_j|+|y-\bar{y}_j|)}{\beta+ |\bar{x}_j-\bar{y}_j|^p} \leq c \left|\begin{pmatrix} x-\bar{x}_j \\ y-\bar{y}_j\end{pmatrix} \right|
\end{align*}
as well as
\begin{align*}
    D_2 &\coloneqq\left(\frac{1}{(\beta+ |x-y|^p)}-\frac{1}{(\beta+ |\bar{x}_j-\bar{y}_j|^p)} \right) \left( \lbrack K_*v \rbrack(x)-\lbrack K_*v \rbrack(y)\right)
    \\ & \leq 2 \|K_*\|_{Y, \mathcal{C}} \left( \frac{|\bar{x}_j-\bar{y}_j|^p-|x-y|^p}{(\beta+ |x-y|^p)(\beta+ |\bar{x}_j-\bar{y}_j|^p)} \right)
    \\ & \leq \frac{2c \|K_*\|_{Y, \mathcal{C}}}{\beta^2} \left|\begin{pmatrix} x-\bar{x}_j \\ y-\bar{y}_j\end{pmatrix} \right|.
\end{align*}
The claimed statement then follows by definition of~$U^1_i(R_1)$ and~$U^2_j(R_1)$ from Lemma~\ref{lem:openneighbor} and noting that
\[
    g(\mu, \bar{\mu}^1_i)=|z-\bar{z}_i| \quad \text{for all}~\mu= \sign(\bar{q}(\bar{z}_i))\delta_z \in U^1_i(R_1)
\]
as well as
\[
    g(\mu, \bar{\mu}^2_j)=\left|\begin{pmatrix} x-\bar{x}_j \\ y-\bar{y}_j\end{pmatrix} \right| \quad \text{for all}~\mu= \mathcal{D}_\beta(x,y) \in U^2_i(R_1).
\]
Since all involved constants are independent of~$i$ and~$j$, respectively, we conclude.
\end{proof}
\begin{prop} \label{prop:quadraticgrwoth}
Let Assumption~$(\mathbf{B3})$ hold. Then there are~$\theta>0 $ and a radius~$0<R_2$ with
\[
    1-\langle \bar{q}, \mu  \rangle \geq \theta\, g(\mu,\bar{\mu}^\ell_j)^2  \quad \text{for all}~\mu \in U^\ell_j(R_2),
\]
and~$j=1,\dots,\bar{N}_\ell$,~$\ell=1,2$.
\end{prop}
\begin{proof}
Since~$\bar{z}_i \in \operatorname{int}\Omega$ is a global extremum of~$\bar{q}$ and~$(\bar{x}_j,\bar{y}_j) \in \operatorname{int}\Omega \times \operatorname{int}\Omega$ is a global maximum of~$\Psi_{\bar{q}}$, we have~$\nabla \bar{q}(\bar{z}_i)=0$ and $\nabla \Psi_{\bar{q}}(\bar{x}_j, \bar{y}_j)=0$, respectively. Using the non-degeneracy of the associated Hessians, see Assumption~$(\mathbf{B3})$, and the continuity of~$\bar{q}$, we conclude the existence of~$R_2>0$ as well as of~$\theta>0$ with
\[
\sign(\bar{q}(z))= \sign(\bar{q}(\bar{z}_i)),~1-|\bar{q}(z)|/\alpha \geq \theta \, |z-\bar{z}_i|^2 \quad \text{for all}~z \in B_{R_2}(\bar{z}_i),
\]
as well as
\[
    1-\Psi_{\bar{q}}(x,y) \geq \theta \, \left|\begin{pmatrix} x-\bar{x}_j \\ y-\bar{y}_j\end{pmatrix} \right|^2  \quad \text{for all}~(x,y) \in B_{R_2}(\bar{x}_j,\bar{y}_j),
\]
by Taylor's expansion. This implies
\[
    1-\langle \bar{q}, \mu_1 \rangle=1-\sign(\bar{q}(z)) \bar{q}(z)/\alpha=1-|\bar{q}(z)|/\alpha \geq \theta \, |z-\bar{z}_i|^2= \theta \,g(\mu_1, \bar{\mu}_i)^2,
\]
as well as
\[
    1-\langle \bar{q}, \mu_2 \rangle=1-\Psi_{\bar{q}}(x,y) \geq 1-|\bar{q}(z)|/\alpha \geq \theta \, \left|\begin{pmatrix} x-\bar{x}_j \\ y-\bar{y}_j\end{pmatrix} \right|^2
\]
for all
\begin{equation} \label{eq:quadhelp}
    \mu_1= (\operatorname{sign}(\bar{q}(\bar{z}_i))/\alpha) \delta_z \in U^1_i(R_2)~\quad \text{and} \quad \mu_2= \mathcal{D}_\beta(x,y) \in U^2_j(R_2).
\end{equation}
By Lemma~\ref{lem:openneighbor}, all elements of~$U^1_i(R_2)$ and~$U^2_i(R_2)$, respectively, are of the form~\eqref{eq:quadhelp}, thus finishing the proof. 
\end{proof}
Summarizing the previous observations, we conclude Theorem~\ref{thm:linear} using the results of from~\cite{linconv}:
\begin{proof}[Proof of Theorem~\ref{thm:linear}]
Summarizing our previous observations, we have that:
\begin{itemize}
    \item The function~$F$ is strongly convex around the optimal observation~$\bar{y}$, see Assumption~$(\mathbf{B2})$.
    \item According to Lemma~\ref{lem:onlyfinite}, there exists~$\{\bar{\mu}_j\}^{\bar{N}}_{j=1} \subset \operatorname{Ext}(B)$ with
    $\max_{\mu \in \mathcal{B}}\langle \bar{q}, \mu \rangle=\{\bar{\mu}_j\}^{\bar{N}}_{j=1}$.
    \item The set~$\{\bar{\mu}_j\}^{\bar{N}}_{j=1}$ is linearly independent, see Assumption~$(\mathbf{B4})$.
    \item The unique solution~$\bar{u}=\sum^{\bar{N}}_{j=1} \bar{\gamma}_j \bar{\mu}_j$ satisfies~$\bar{\gamma}_j>0$, see Assumption~$(\mathbf{B5})$.
    \item There are~$d_{\mathcal{B}}$-neighborhoods~$U_j$ of~$\bar{\mu}_j$ for~$j=1,\dots,\bar{N}$, a function~$g \colon \operatorname{Ext}(B) \times \operatorname{Ext}(B) \to \mathbb{R}$ and $C_K,\theta >0$ with
    \[
        \|K(\mu-\bar{\mu}_j)\|_Y \leq  C_K \, g(\mu, \bar{\mu}_j),~ 1-\langle \bar{q}, \mu \rangle \geq \theta\, g(\mu,\bar{\mu}_j)^2 \quad \text{for all}~\mu \in U_j \cap \operatorname{Ext}(B).
    \]
\end{itemize}
Consequently, the assumptions of \cite[Theorem~3.8]{linconv} are satisfied, and applying it we conclude the linear convergence of Theorem~\ref{thm:linear}.
\end{proof}

\small
\bibliographystyle{plain}
\bibliography{kr}

\end{document}